\documentclass[12pt,a4paper]{article}
\usepackage{arxiv}
\usepackage[utf8]{inputenc} 
\usepackage[T1]{fontenc}    
\usepackage{hyperref}       
\usepackage{url}            
\usepackage{booktabs}       
\usepackage{amsfonts}       
\usepackage{nicefrac}       
\usepackage{microtype}      
\usepackage{lipsum}		    
\usepackage{natbib}         
\usepackage{fancyhdr}       
\usepackage{color}          
\usepackage{graphicx}       
\usepackage{amsthm}
\usepackage{amssymb}
\usepackage{pstricks}
\usepackage{amsmath}
\usepackage{dsfont}

\renewcommand{\cite}{\citet}
\theoremstyle{plain}
\newtheorem{theorem}{Theorem}[section]
\newtheorem{proposition}[theorem]{Proposition}
\newtheorem{lemma}[theorem]{Lemma}
\newtheorem{corollary}[theorem]{Corollary}

\theoremstyle{definition}
\newtheorem{definition}[theorem]{Definition}
\theoremstyle{remark}
\newtheorem{remark}[theorem]{Remark}

\makeatletter \@addtoreset{equation}{section}
\makeatother

\newcommand{\N}{\mathbb{Z}_{+}}

\newcommand{\R}{\mathbb{R}}
\title{Existence and uniqueness for reflected BSDE with multivariate point process and right upper-semi-continuous obstacle}

\date{} 					

\author{
  Baadi Brahim  \\
 Ibn Tofa\"{\i}l University,\\
  Department of mathematics, faculty of sciences,\\
  BP 133, K\'{e}nitra, Morocco \\
  \texttt{brahim.baadi@uit.ac.ma}\\
   \And
 Mohamed Marzougue \\
  Abdelmalek Essaadi University ,\\
   Department of mathematics, Faculty of Sciences,\\
  LaR2A Laboratory\\
  Tetouan, Morocco\\
 \texttt{m.marzougue@uae.ac.ma} \\
}

\begin{document}
\maketitle

\begin{abstract}
In a  noise driving by a  multivariate point process $\mu$ with predictable compensator $\nu$, we prove existence and uniqueness of the reflected backward stochastic differential equation's solution with a lower obstacle $(\xi_{t})_{t\in[0,T]}$ which is assumed to be right upper-semicontinuous but not necessarily right-continuous process and a Lipschitz driver $f$. The result is established by using Mertens decomposition of optional strong (but not necessarily right continuous) super-martingales, an appropriate generalization of It\^{o}'s formula due to Gal'chouk and Lenglart and some tools from optimal stopping theory. A comparison theorem for this type of equations is given.
\end{abstract}

\section{Introduction}

Nonlinear backward stochastic differential equations (BSDEs in short) were introduced by \cite{PP1990} where the authors have established the existence and uniqueness of the solution to multidimensional such equations in a finite time interval under uniform Lipschitz coefficient. Up to now, BSDEs have been successfully applied to many various areas such as mathematical finance \cite{EQ:1997,ELKarouiPeng1997}, stochastic control and stochastic games \cite{HL,EH:2003}, partial differential equations \cite{BBP:1997,PP1992} and so on.

As a various of BSDEs, \cite{ELKaroui1997} have introduced reflected backward stochastic differential equations (Reflected BSDEs in short), which are BSDEs but the solution is forced to stay above an adapted continuous process called obstacle (or barrier).
Many efforts have been made to relax the regularity of the obstacle process and/or to consider a larger stochastic basis than the Brownian one, in this context \cite{Hamadene2002} has considered the Reflected BSDEs when the obstacle is discontinuous in Brownian setting, further, \cite{HamadeneOknine2003} have studied the Reflected BSDEs with jumps in Poisson-Brownian setting where the authors have proved the existence and uniqueness of solutions when the obstacle is assumed to be right
continuous with left limits (rcll) whose jumping times are inaccessible stopping times, and thus the jumping times of the $Y$-component of the solution come only from those of its Poisson process part and then they are inaccessible. Later, \cite{Essaky2008} and \cite{HamadeneOknine2011} have studied the case when the obstacle is only rcll and the jumping times of the $Y$-component of the solution come not only from those of its Poisson process part but also from those of the obstacle which are predictable stopping times. Recently, \cite{foresta2019} has studied Reflected BSDEs when the noise is driven by a marked point process and a Brownian motion, and the obstacle process is supposed rcll.

Recently, \cite{Ouknine2015} have introduced a new extension of Reflected BSDEs to the case where the obstacle is not necessarily right-continuous in which the additional nondecreasing process, which pushes the $Y$-component of the solution to stay above the obstacle, is no longer right continuous, and thus also the solution is no longer right continuous. The authors have showed the existence and uniqueness of the solution to Reflected BSDEs where the obstacle is assumed to be right upper semi-continuous. Later, many papers have discussed the problem of Reflected BSDEs beyond right-continuity (see for instance \cite{akdim,baadi2017,baadi2018,GIOQ,Klimsiak2018,Ma,ME:2019,ME:2020} in different directions).

As in \cite{Ouknine2015} an independent Poisson random measure is added to the driving noise, and motivated by several applications to stochastic optimal control and financial modelling, more general multivariate point processes were considered in the BSDE. For this end, we mainly consider Reflected BSDEs driven by a multivariate point measure when the obstacle process is assumed to be right upper semi-continuous. Precisely, consider a multivariate point measure $\mu$ with predictable compensator $\nu$, and given a data $(\xi,f)$ for r.u.s.c obstacle $\xi$ and Lipschitz driver $f$. We address the following Reflected BSDE
$$
Y_{t} =
\xi_{T}+\int_{t}^{T}f(s,Y_{s},Z_{s}(.))ds-\int_{t}^{T}\int_{\mathcal{U}}Z_{s}(x)(\mu-\nu)(dx,ds)
+A_{T}-A_{t}+C_{T-}-C_{t-}\quad 0\leq t\leq T
$$
where $(Y,Z,A,C)$ is called the solution of Reflected BSDEs associated with parameters $(\xi,f)$ such that $Y$ dominates the obstacle $\xi$, $A$ is a nondecreasing right-continuous predictable process which satisfies some Skorokhod conditions and $C$ is a nondecreasing right-continuous adapted purely discontinuous process which satisfies some minimality conditions. The role of $A$ and $C$ is to control the left-jumps and right-jumps of $Y$ respectively.

The paper is organized as follows: In section \ref{sec1}, we give some notations, preliminary results and assumptions needed in this article. The section \ref{sec2} is devoted to prove our main results as well as existence, uniqueness and comparison theorem for the solutions of such equations. In section \ref{sec3} we give some connections between RBSDEs studied in the section \ref{sec2} with an optimal stopping problem.

\section{Preliminaries}\label{sec1}

Consider a fixed positive real number $T>0$, a Lusin\footnote{this is a sequence $(S_n,X_n)$ of points, with distinct times of occurrence $S_n$ and with marks $X_n$, so it can be viewed as a random measure of the form $$\mu(dt,dx)=\sum_{n\geq 1:S_{n}}\varepsilon_{(S_{n},X_{n})}(dt,dx),$$
where $\varepsilon_{(t,x)}$ denotes the Dirac measure} space $(\mathcal{U},\mathcal{E})$ and a
filtered probability space $(\Omega,\mathcal{F},\mathbb{P},\mathbb{F}=\{\mathcal{F}_{t}, t\geq0\})$, with  $(\mathcal{F}_{t})_{(t\geq0)}$  complete, right continuous and quasi-left continuous. We denote by $\mathcal{P}$ the predictable $\sigma$-field on $\Omega\times[0,T]$, and for any auxiliary measurable space $(G,\mathcal{G})$ a function on the product $\Omega\times[0,T]\times G$ which is measurable with respect to $\mathcal{P}\otimes\mathcal{G}$  is called predictable.\\

Let $\mu$ be an integer-valued random measure on $\R_{+}\times\mathcal{U}$. In the sequel we use a martingale representation theorem for the random measure $\mu$, namely lemma$~\ref{lemme1}$. For this reason, we suppose that  $(\mathcal{F}_{t})_{(t\geq 0)}$ is the natural filtration of $\mu$, i.e. the smallest right-continuous filtration in which $\mu$ is optional. We also assume that $\mu$ is a
discrete random measure, i.e. the sections of the set $D=\{(\omega,t);\mu(\omega,\{t\}\times \mathcal{U})=1\}$
are finite on every finite interval.\\
We denote by $\nu$ the predictable compensator of the measure $\mu$, relative
to the filtration $(\mathcal{F}_{t})_{(t\geq 0)}$.  Then, $\nu$ can be disintegrated as follows
\begin{equation}
\nu(\omega,dt,dx)=dK_{t}(\omega)\phi_{\omega,t}(dx),
 \label{m}
\end{equation}
where $K$ is a right-continuous nondecreasing predictable process such that $K_{0}=0$, and
$\phi$ is a transition probability from $(\Omega\times[0,T],\mathcal{P})$ into $(\mathcal{U},\mathcal{E})$. We suppose, also, that $\nu$  satisfies $\nu(\{t\}\times dx)\leq1$ identically, so that $\Delta K_{t}=K_{t}-K_{t-}\leq1$.\\

We denote by $\mathcal{B}(\mathcal{U})$ the set of all Borel measurable functions on  $\mathcal{U}$. Given a measurable function $Z:\Omega\times[0,T]\times \mathcal{U}\rightarrow \R$, we write $Z_{\omega,t}(x)= Z(\omega,t,x)$, so that $Z_{\omega,t}$ often abbreviated as $Z_{t}$ or $Z_{t}(.)$, is an element of $\mathcal{B}(\mathcal{U})$. In this paper for a given  $\beta\geq 0$,  we denote:
\begin{itemize}
\item $\mathcal{T}_{t,T}$ (resp. $\mathcal{T}^p_{t,T}$) is the set of all stopping times (resp. predictable stopping times) $\tau$
such that $\mathbb{P}(t\leq \tau \leq T)=1$. More generally, for a
given stopping time  $S$ in $\mathcal{T}_{0,T}$, we denote by
   $\mathcal{T}_{S,T}$ the set of all  stopping times $\tau$
  such that $\mathbb{P}(S \leq \tau \leq T)=1$.
\item  $\mathds{S}^{2,\beta}$ is the set of real-valued optional processes $Y$ such
that:
         $$\||Y\||_{\mathds{S}^{2,\beta}}^{2}=E\Bigl[ess\sup_{\tau\in\mathcal{T}_{0,T}}e^{\beta \tau}|Y_{\tau}|^{2}\Bigr]<\infty.\footnote{The map $\||.\||_{\mathds{S}^{2,0}}$ is a norm on $\mathds{S}^{2,0}$. In particular, if $Y\in\mathds{S}^{2,0}$ is such that
$\||Y\||_{\mathds{S}^{2,0}}=0$, then $Y$ is indistinguishable from the null process, that is $Y_{t}=0$, $0\leq t\leq T$ a.s. Moreover when $Y$ is right-continuous, $\||Y\||_{\mathds{S}^{2,\beta}}^{2}=E\Bigl[\sup_{t\in[0,T]}e^{\beta t}|Y_{t}|^{2}\Bigr]$ (cf. Proposition$~\ref{propA1}$)}$$

\item $\mathds{H}^{2,\beta}$ is the set of real-valued predictable processes $Z$ such that
         $$\|Z\|_{\mathds{H}^{2,\beta}}^{2}=E\Bigl[\int_{0}^{T}e^{\beta t} \int_{\mathcal{U}}|Z_{t}(x)-\bar{Z}_{t}|^{2}\nu(dx,dt)+\sum_{0< t\leq T}e^{\beta t}|\bar{Z}_{t}|^{2}(1-\triangle K_{t})\Bigr]<\infty.$$
where $\bar{Z}_{t}=\int_{\mathcal{U}}Z_{t}(x)\nu(\{t\},dx)$, $ 0\leq t \leq T$.
\end{itemize}

\begin{remark}
\label{Banach}
\begin{itemize}
  \item The space $\mathds{S}^{2,\beta}$ endowed with the norm
$\parallel\mid.\parallel\mid_{\mathds{S}^{2,\beta}}$ is a Banach space ( \cite[Proposition 2.1]{Ouknine2015}).
  \item Let $\beta\geq0$ so that $e^{\beta t}\geq 1$, for a given  predictable processes $Z$, we have
$$
E\Bigl[\int_{0}^{T}\int_{\mathcal{U}}|Z_{t}(x)-\bar{Z}_{t}|^{2}\nu(dx,dt)+\sum_{0< t\leq T}|\bar{Z}_{t}|^{2}(1-\triangle K_{t})\Bigr]=\|Z\|_{\mathds{H}^{2,0}}^{2}\leq \|Z\|_{\mathds{H}^{2,\beta}}^{2}
$$
\end{itemize}
\end{remark}
Now we enumerate an important proposition, which can be found in \cite[Theorem IV.84]{DellacherieMeyer1975} (or in \cite[Theorem 3.2.]{AshkanNikeghbali2006}).
\begin{proposition}
\label{AshkanNikeghbali}
\begin{itemize}
  \item Let $(X_t)$ and $(Y_t)$ be two optional processes. If for every
finite stopping time $\tau$ one has, $X_{\tau}=Y_{\tau}$, then the
processes $(X_t)$ and $(Y_t)$ are indistinguishable.
  \item Let $X$ and $Y$ be two optional processes such that
$X_{S}\leq Y_{S}$ a.s. for all $S\in\mathcal{T}_{0,T}$. Then, $X\leq
Y$ up to an evanescent set.
\end{itemize}
\end{proposition}

We recall the following  decomposition property of martingales  (or local martingale) as the integral of a $\mathcal{P}$-measurable process defined on  $\Omega\times[0,T]\times \mathcal{U}$, with respect to the random measure $\mu-\nu$. More precisely we have the two results:
\begin{lemma} \textbf{[\cite[Theorem 5.4]{Jacod1975}]}
$\label{lemme1}$ Each martingale (or local martingale) ($M_{t}$)
can be written as an integral
\begin{equation}
M_{t}=M_{0}+\int_{0}^{t}\int_{\mathcal{U}}Z_{s}(x)(\mu-\nu)(ds,dx)
 \label{mono3}
\end{equation}
where $Z$ is a predictable process defined on $\Omega\times[0,T]\times \mathcal{U}$.
\end{lemma}
This result is similar to the known representation of martingales as the stochastic integral of a predictable process with respect to a fundamental martingale, for Poisson and Wiener processes.
\begin{proposition} \textbf{[\cite[Proposition 3.71-(a)]{Jacod1975}]}
Let $M$ the martingale of the form $~\ref{mono3}$, then  the predictable bracket $<M,M>_{t}$ is given by:
$$<M,M>_{T}=\int_{0}^{T}\int_{\mathcal{U}}|Z_{t}(x)-\bar{Z}_{t}|^{2}\nu(dx,dt)+\sum_{0< t\leq T}|\bar{Z}_{t}|^{2}(1-\triangle K_{t})$$
where $\bar{Z}_{t}=\int_{\mathcal{U}}Z_{t}(x)\nu(\{t\},dx)$, $ 0\leq t \leq T$.
\end{proposition}

The random variable $\xi$ is $\mathcal{F}_{T}$-measurable with values in $\R$ and
$f$  a real-valued function on $\Omega\times[0,T]\times\R\times\mathcal{B}(\mathcal{U})$, such that
$f(\omega,t,y,Z_{t}(.))$ is predictable for any predictable function $Z$ on $\Omega\times[0,T]\times\mathcal{U}$. By $Prog$ we denote the $\sigma$-field of progressive subsets of $\Omega\times[0,T]$.
\begin{definition}[\textbf{Driver, Lipschitz driver}]
A function $f$ is said to be a driver if:
\begin{itemize}
    \item $f:\Omega\times [0,T]\times\R\times\mathcal{B}(\mathcal{U})\longrightarrow \R$ \\
    $(\omega,t,y,z)\longmapsto f(\omega,t,y,z)$ is $Prog\times\mathcal{B}(\R)\times\mathcal{E}$-measurable.
    \item $E[\int_{0}^{T}\mid f(t,0,0)\mid^{2}dt]<\infty$.
\end{itemize}
A driver $f$ is called a Lipschitz driver if moreover there exists a
constant $L\geq0$ such that $\mathbb{P}\otimes dt$-a.s., for each
$y,y'\in\R$ and $z,z'\in L(\mathcal{U},\mathcal{E},\phi_{\omega,t}(dx))$

\begin{multline}
\Bigl|f(\omega,t,y,z)-f(\omega,t,y',z')\Bigr|\leq
L\Bigl(|y-y'|+ \Bigl(\int_{\mathcal{U}}|z(x)-z'(x)- \int_{\mathcal{U}}(z(u)-z'(u))\phi_{\omega,t}(du)|^{2}\phi_{\omega,t}(dx) \\ +(1-\triangle K_{t})|\int_{\mathcal{U}}(z(x)-z'(x))\phi_{\omega,t}(dx)|^{2}\Bigr)^{\frac{1}{2}}\Bigr)
 \label{moneq0}
\end{multline}
\end{definition}
For a l\`{a}dl\`{a}g process $Y$, we denote by $Y_{t+}$ and
$Y_{t-}$ the right-hand and left-hand limit of $Y$ at  $t$. We
denote by $\Delta_{+}Y_{t}=Y_{t+}-Y_{t}$ the size of the
right jump of $Y$ at $t$, and by
$\Delta Y_{t}=Y_{t}-Y_{t-}$ the size of the left jump of
$Y$ at $t$.

\section{Reflected BSDE on optional obstacle and multivariate point process}\label{sec2}

Let $f$ be a Lipschitz driver and $\xi=(\xi_t)_{t\in[0,T]}$ be a left-limited process belongs to $\mathds{S}^{2,0}$. We suppose moreover that the process $\xi$ is right upper-semicontinuous (r.u.s.c. for short). A process $\xi$
satisfying the previous properties will be called an obstacle, or a barrier.
\begin{definition}
\label{definition1} A process $(Y,Z,A,C)$ is said to be a
solution to the reflected BSDE with parameters $(f,\xi)$, where $f$
is a Lipschitz driver and $\xi$ is an obstacle, if
\begin{itemize}
  \item $(Y,Z,A,C) \in\mathds{S}^{2,0}\times\mathds{H}^{2,0}\times\mathds{S}^{2,0}\times\mathds{S}^{2,0}$ and  for all $\tau\in\mathcal{T}_{0,T}$:
  \begin{multline}
Y_{\tau} =
\xi_{T}+\int_{\tau}^{T}f(s,Y_{s},Z_{s}(.))ds-\int_{\tau}^{T}\int_{\mathcal{U}}Z_{s}(x)(\mu-\nu)(dx,ds)
+A_{T}-A_{\tau}+C_{T-}-C_{\tau-}
   \label{moneq4}
\end{multline}
  \item  $Y\geq \xi$ (up to an evanescent set) a.s.
  \item The process $A$ is a nondecreasing right-continuous predictable
process with $A_0=0$, $E(A_T)<\infty$  such that:
\begin{equation}
\int_{0}^{T}1_{ \{Y_{t}>\xi_{t}\}}dA_{t}^{c}=0 \,\,\  a.s.\,\,\ and\,\,\ (Y_{\tau-}-
\xi_{\tau-})(A_{\tau}^{d}-A_{\tau-}^{d})=0 \,\,\  a.s. \,\,\ \forall \tau\in\mathcal{T}^p_{0,T}
 \label{moneq6}
\end{equation}
  \item The process $C$ is a nondecreasing right-continuous adapted purely
discontinuous process with  $C_{0-}=0$, $E(C_T)<\infty$
such that:
\begin{equation}
 (Y_{\tau}-\xi_{\tau})(C_{\tau}-C_{\tau-})=0 \,\,\
a.s.\,\,\ \forall \tau\in\mathcal{T}_{0,T}
 \label{moneq7}
\end{equation}
\end{itemize}
Here $A^{c}$  denotes the continuous part of the nondecreasing
process $A$ and $A^{d}$ its discontinuous part.
\end{definition}

\begin{remark}
\begin{itemize}
  \item If the process $(Y,Z,A,C)$ satisfies equation $~\ref{moneq4}$  then  for all  $t\in[0,T]$, a.s.
\begin{eqnarray*}
  Y_{t} = \xi_{T}+\int_{t}^{T}f(s,Y_{s},Z_{s}(.))ds-\int_{t}^{T}\int_{\mathcal{U}}Z_{s}(x)(\mu-\nu)(dx,ds)+A_{T}-A_{t}+C_{T-}-C_{t-}
\end{eqnarray*}
  \item  If the process $(Y,Z,A,C)$ satisfies the above definition, then the process $Y$ has left and right limits. Moreover, the process $(Y_{t}+\int_{0}^{t}f(s,Y_{s},Z_{s}(.))ds)_{t\in[0,T]}$ is a strong supermartingale.
\item From $~\ref{moneq4}$, we get $\triangle C_{t}=C_{t}-C_{t-}=Y_{t}-Y_{t+}$.
Hence, $Y_{t}\geq Y_{t+}$, for all $t\in[0,T]$, which implies that $Y$ is necessarily r.u.s.c. Moreover,
$Y$ is right-continuous if and only if $C\equiv0$.
\end{itemize}
\end{remark}

\subsection{The case when the driver $f$ does not depend on $(y,z)$.}

In this part, we assume that the driver $f$ depends only on $\omega$ and $s$ (i.e. $f(s,y,z)=f(\omega,s)$). Let us first prove the following Lemma which will play an important role in the sequel.
\begin{lemma}
\label{lemme2} Let
$(Y^{1},Z^{1},A^{1},C^{1})$ (resp. $(Y^{2},Z^{2},A^{2},C^{2})$.)
be a solution to the reflected BSDE associated with driver $f_{1}(\omega,s)$
(resp.$f_{2}(\omega,s)$) and with obstacle $\xi$. Then there exists $c>0$
such that for all $\epsilon>0$ and all $\beta>\frac{1}{\epsilon^{2}}$ we have
\begin{equation}
\|Z^{1}-Z^{2}\|^{2}_{\mathds{H}^{2,\beta}}\leq \epsilon^{2} E\Bigl[\int_{0}^{T}e^{\beta s}(f_{1}(s)-f_{2}(s))^{2}ds\Bigr]
 \label{unicite1}
\end{equation}
and
\begin{equation}
 \||Y^{1}-Y^{2}|\|_{\mathds{S}^{2,\beta}}^{2}\leq2\epsilon^{2}(1+2c^{2}) E\Bigl[\int_{0}^{T}e^{\beta s}(f_{1}(s)-f_{2}(s))^{2}ds\Bigr].
 \label{unicite2}
 \end{equation}
\end{lemma}

\begin{proof}
Let  $\beta >0$ and   $\epsilon>0$ be such that $\beta\geq
\frac{1}{\epsilon^{2}}$. We set  $\widetilde{Y}:=Y^{1}-Y^{2}$,
$\widetilde{Z}:=Z^{1}-Z^{2}$, $\widetilde{A}:=A^{1}-A^{2}$,
$\widetilde{C}:=C^{1}-C^{2}$ and
$\widetilde{f}(\omega,t):=f_{1}(\omega,t)-f_{2}(\omega,t)$. We note
that  $\widetilde{Y}_{T}:=\xi_{T}-\xi_{T}=0$. Moreover,
\begin{equation}
\widetilde{Y}_{\tau}=\widetilde{Y}_{0}-\int_{0}^{\tau}\widetilde{f}(s)ds+
\int_{0}^{\tau}\int_{\mathcal{U}}\widetilde{Z}_{s}(x)(\mu-\nu)(dx,ds)-\widetilde{A}_{\tau}-\widetilde{C}_{\tau-}\,\,\
a.s.\,\,\ \forall \tau\in\mathcal{T}_{0,T} \label{moneq9}
\end{equation}

since
$\widetilde{A}_{0}=A_{0}^{1}-A_{0}^{2}=0$ and
$\widetilde{C}_{0-}=C_{0-}^{1}-C_{0-}^{2}=0$. Thus we see that
$\widetilde{Y}$  is an optional (strong) semimartingale with
decomposition
$\widetilde{Y}_{t}=\widetilde{Y}_{0}+M_{t}+A_{t}+B_{t}$, where
$M_{t}=\int_{0}^{t}\int_{\mathcal{U}}\widetilde{Z}_{s}(x)(\mu-\nu)(dx,ds)$,
$A_{t}=-\int_{0}^{t}\widetilde{f}(s)ds-\widetilde{A}_{t}$ and
$B_{t}=-\widetilde{C}_{t-}$ (the notation is that of Corollary$~\ref{corollaire1}$). Applying  the generalization of the classical It\^{o} formula  (see Theorem$~\ref{galchouklenglart}$ and Corollary$~\ref{corollaire1}$) to $F(x,y)=xy^{2}$ with  $x= e^{\beta t}$ and  $y=\widetilde{Y}_{t}$ gives: almost surely, for all $t\in[0,T]$,
\begin{eqnarray*}
  e^{\beta t}\widetilde{Y}_{t}^{2} &=&  -\int_{0}^{t}\beta e^{\beta s}\widetilde{Y}_{s}^{2}ds+2\int_{0}^{t}e^{\beta s}\widetilde{Y}_{s-}d(A+M)_{s}- \int_{0}^{t} e^{\beta s}d<M,M>_{s} \\
   &-& \sum_{0< s\leq t}e^{\beta s}(\widetilde{Y}_{s}-\widetilde{Y}_{s-})^{2}-\int_{0}^{t}2e^{\beta
s}\widetilde{Y}_{s}d(B)_{s+}-\sum_{0\leq s<t}e^{\beta
  s}(\widetilde{Y}_{s+}-\widetilde{Y}_{s})^{2}.
\end{eqnarray*}
Using the expressions of $M$, $A$ and $B$  and the fact that
$\widetilde{Y}_{T}=0$, we get: almost surely, for all $t\in[0,T]$,
\begin{multline}
\label{mono8}
   e^{\beta t}\widetilde{Y}_{t}^{2} +
   \int_{t}^{T}\int_{\mathcal{U}}e^{\beta s}|\widetilde{Z}_{s}(x)-\hat{Z}_{s}|^{2}\nu(dx,ds)+\sum_{t< s\leq T}e^{\beta s}|\hat{Z}_{s}|^{2}(1-\triangle K_{s}) = -\int_{t}^{T}\beta e^{\beta s}\widetilde{Y}_{s}^{2}ds\\
   +2\int_{t}^{T}e^{\beta s}\widetilde{Y}_{s-}\widetilde{f}(s)ds
  +2\int_{t}^{T}e^{\beta s}\widetilde{Y}_{s-}d\widetilde{A} - 2\int_{t}^{T}\int_{\mathcal{U}}e^{\beta s}\widetilde{Y}_{s-}\widetilde{Z}_{s}(x)(\mu-\nu)(dx,ds)\\
 +2\int_{t}^{T}e^{\beta s}\widetilde{Y}_{s}d(\widetilde{C})_{s}
  -\sum_{t<s\leq T}e^{\beta s}(\widetilde{Y}_{s}-\widetilde{Y}_{s-})^{2}
-\sum_{t\leq s<T}e^{\beta s}(\widetilde{Y}_{s+}-\widetilde{Y}_{s})^{2}
\end{multline}
where $\hat{Z}_{t}=\int_{\mathcal{U}}\widetilde{Z}_{t}(x)\nu(\{t\},dx)$. It is clear that for all $t\in[0,T]$  $-\sum_{t<s\leq T}e^{\beta
s}(\widetilde{Y}_{s}-\widetilde{Y}_{s-})^{2}-\sum_{t\leq
s<T}e^{\beta s}(\widetilde{Y}_{s+}-\widetilde{Y}_{s})^{2} \leq 0$.
By applying the inequality $2ab\leq
(\frac{a}{\epsilon})^{2}+\epsilon^{2}b^{2}$ for all $(a,b)$
in $\mathbb{R}^{2}$, we get: a.s. for all $t\in[0,T]$
\begin{eqnarray*}
  -\int_{t}^{T}\beta e^{\beta s}\widetilde{Y}_{s}^{2}ds+2\int_{t}^{T}e^{\beta
   s}\widetilde{Y}_{s-}\widetilde{f}(s)ds \leq (\frac{1}{\epsilon^{2}}-\beta)\int_{t}^{T}e^{\beta
   s}\widetilde{Y}_{s-}^{2}ds+\epsilon^{2}\int_{t}^{T}e^{\beta
   s}\widetilde{f}(s)^{2}ds.
\end{eqnarray*}
As $\beta>\frac{1}{\epsilon^{2}}$, we have $(\frac{1}{\epsilon^{2}}-\beta)\int_{t}^{T}e^{\beta s}\widetilde{Y}_{s-}^{2}ds\leq 0$, for all $t\in[0,T]$ a.s. Similarly to the proof of Lemma 3.4. in  \cite{baadi2017} and from the properties $~\ref{moneq7}$ and $~\ref{moneq6}$  of the definition$~\ref{definition1}$, the terms $\int_{t}^{T}e^{\beta s}\widetilde{Y}_{s-}d\widetilde{A}$ and $\int_{t}^{T}e^{\beta s}\widetilde{Y}_{s}d(\widetilde{C})_{s}$ are  non-positive. The above
observations, together with equation $~\ref{mono8}$, lead to the following inequality: a.s., for all

\begin{multline}
\label{mono9}
   e^{\beta t}\widetilde{Y}_{t}^{2} + \int_{t}^{T}\int_{\mathcal{U}}e^{\beta s}|\widetilde{Z}_{s}(x)-\hat{Z}_{s}|^{2}\nu(dx,ds)+\sum_{t< s\leq T}e^{\beta s}|\hat{Z}_{s}|^{2}(1-\triangle K_{s})\leq \epsilon^{2}\int_{t}^{T}e^{\beta
   s}\widetilde{f}(s)^{2}ds\\
  -2\int_{t}^{T}\int_{\mathcal{U}}e^{\beta s}\widetilde{Y}_{s-}\widetilde{Z}_{s}(x)(\mu-\nu)(dx,ds), \,\,\,\,\,\,\ a.s. \,\,\forall t\in[0,T].
\end{multline}
From the  inequality $~\ref{mono9}$ we derive first an estimate for  $\|\widetilde{Z}(x)\|_{\mathds{H}^{2,\beta}}$ and then an estimate for $\||\widetilde{Y}\||_{\mathds{S}^{2,\beta}}$.\\

To this purpose, we show first  that the stochastic integral $\int_{0}^{T}\int_{\mathcal{U}}e^{\beta s}\widetilde{Y}_{s-}\widetilde{Z}_{s}(x)(\mu-\nu)(dx,ds)$ has zero expectation. For that, we show that
\begin{equation}
E\Bigl[\sqrt{\int_{0}^{T}\int_{\mathcal{U}}e^{2\beta
s}\widetilde{Y}_{s-}^{2}|\widetilde{Z}_{s}(x)-\hat{Z}_{s}|^{2}\nu(dx,ds)+\sum_{0< s\leq T}e^{2\beta
s}\widetilde{Y}_{s-}^{2}|\hat{Z}_{s}|^{2}(1-\triangle K_{s})}\Bigr]< \infty.
 \label{mono10}
\end{equation}
By using the left-continuity of a.e.
trajectory of the process  $(\widetilde{Y}_{s-})$, we have
\begin{equation}
(\widetilde{Y}_{s-})^{2}(\omega)\leq
\sup_{t\in\mathbb{Q}}(\widetilde{Y}_{t-})^{2}(\omega) \,\,\
\text{for all} \,\,\ s\in(0,T],\,\,\  \text{for a.s.}\,\,\
\omega\in\Omega.
 \label{mono12}
\end{equation}
On the other hand, for all $t\in(0,T]$, a.s.,
$(\widetilde{Y}_{t-})^{2} \leq ess\sup_{\tau\in\mathcal{T}_{0,T}}(\widetilde{Y}_{\tau})^{2}$. Then
\begin{equation}
\sup_{t\in\mathbb{Q}}(\widetilde{Y}_{t-})^{2}\leq
ess\sup_{\tau\in\mathcal{T}_{0,T}}(\widetilde{Y}_{\tau})^{2} \,\,\
a.s.
 \label{mono13}
\end{equation}
According to $~\eqref{mono12}$ and $~\eqref{mono13}$ we obtain
\begin{multline}
\label{mono14}
\int_{0}^{T}\int_{\mathcal{U}}e^{2\beta
s}\widetilde{Y}_{s-}^{2}|\widetilde{Z}_{s}(x)-\hat{Z}_{s}|^{2}\nu(dx,ds)+\sum_{0< s\leq T}e^{2\beta
s}\widetilde{Y}_{s-}^{2}|\hat{Z}_{s}|^{2}(1-\triangle K_{s})\leq \\
\int_{0}^{T}\int_{\mathcal{U}}\sup_{t\in\mathbb{Q}}(\widetilde{Y}_{t-})^{2}e^{2\beta
s}|\widetilde{Z}_{s}(x)-\hat{Z}_{s}|^{2}\nu(dx,ds)+\sum_{0< s\leq T}\sup_{t\in\mathbb{Q}}(\widetilde{Y}_{t-})^{2}e^{2\beta
s}|\hat{Z}_{s}|^{2}(1-\triangle K_{s})\leq \\
\int_{0}^{T}\int_{\mathcal{U}}ess\sup_{\tau\in\mathcal{T}_{0,T}}(\widetilde{Y}_{\tau})^{2}e^{2\beta
s}|\widetilde{Z}_{s}(x)-\hat{Z}_{s}|^{2}\nu(dx,ds)+\sum_{0< s\leq T} ess\sup_{\tau\in\mathcal{T}_{0,T}}(\widetilde{Y}_{\tau})^{2}e^{2\beta
s}|\hat{Z}_{s}|^{2}(1-\triangle K_{s}).
\end{multline}

Using $~\eqref{mono14}$, together with Cauchy-Schwarz inequality, we get

\begin{eqnarray*}
E\Bigl[\sqrt{\int_{0}^{T}\int_{\mathcal{U}}e^{2\beta s}\widetilde{Y}_{s-}^{2}|\widetilde{Z}_{s}(x)-\hat{Z}_{s}|^{2}\nu(dx,ds)+\sum_{0< s\leq T}e^{2\beta
s}\widetilde{Y}_{s-}^{2}|\hat{Z}_{s}|^{2}(1-\triangle K_{s})}\Bigr]\leq\\
E\Bigl[\sqrt{ess\sup_{\tau\in\mathcal{T}_{0,T}}(\widetilde{Y}_{\tau})^{2}}\sqrt{\int_{0}^{T}\int_{\mathcal{U}}e^{2\beta
s}|\widetilde{Z}_{s}(x)-\hat{Z}_{s}|^{2}\nu(dx,ds)+\sum_{0< t\leq T}e^{2\beta s}|\hat{Z}_{s}|^{2}(1-\triangle K_{s})}\Bigr]\leq \\  \||\widetilde{Y}\||_{\mathds{S}^{2,0}}.\|\widetilde{Z}\|_{\mathds{H}^{2,2\beta}}< \infty.
\end{eqnarray*}
We conclude that $~\eqref{mono10}$ hods, whence, we get $E\Bigl[\int_{0}^{T}\int_{\mathcal{U}}e^{\beta s}\widetilde{Y}_{s-}\widetilde{Z}_{s}(x)(\mu-\nu)(dx,ds)\Bigr]=0$. By taking expectations on both sides of $~\eqref{mono9}$ with $t=0$, we obtain:

$$
E\Bigl[\int_{0}^{T}\int_{\mathcal{U}}e^{\beta s}|\widetilde{Z}_{s}(x)-\hat{Z}_{s}|^{2}\nu(dx,ds)+\sum_{0< s\leq T}e^{\beta s}|\hat{Z}_{s}|^{2}(1-\triangle K_{s})\Bigr]\leq \epsilon^{2}E\Bigl[\int_{0}^{T}e^{\beta s}\widetilde{f}(s)^{2}ds\Bigr].
$$
Hence, we obtain the first inequality of the lemma: $\|\widetilde{Z}(x)\|_{\mathds{H}^{2,\beta}}\leq \epsilon^{2}E\Bigl[\int_{0}^{T}e^{\beta s}\widetilde{f}(s)^{2}ds\Bigr]$. From $~\eqref{mono9}$ we also get, for all $\tau\in\mathcal{T}_{0,T}$

\begin{equation}
e^{\beta \tau}\widetilde{Y}_{\tau}^{2}\leq \epsilon^{2}\int_{0}^{T}e^{\beta s}\widetilde{f}(s)^{2}ds -2\int_{\tau}^{T}\int_{\mathcal{U}}e^{\beta s}\widetilde{Y}_{s-}\widetilde{Z}_{s}(x)(\mu-\nu)(dx,ds),  \,\,\,\,\ a.s.
\label{mono15}
\end{equation}

By taking first the essential supremum over $\tau\in\mathcal{T}_{0,T}$, and then the expectation on
both sides of the inequality $~\eqref{mono15}$, we obtain:

\begin{multline}
\label{mono16}
   E\Bigl[ess\sup_{\tau\in\mathcal{T}_{0,T}}e^{\beta \tau}\widetilde{Y}_{\tau}^{2}\Bigr]\leq \epsilon^{2} E\Bigl[\int_{0}^{T}e^{\beta s}\widetilde{f}(s)^{2}ds\Bigr]+2E\Bigl[ess\sup_{\tau\in\mathcal{T}_{0,T}}|\int_{0}^{\tau}\int_{\mathcal{U}}e^{\beta s}\widetilde{Y}_{s-}\widetilde{Z}_{s}(x)(\mu-\nu)(dx,ds)|\Bigr].
\end{multline}

Let us consider the last term in  $~\eqref{mono15}$. By applying the Proposition$~\eqref{propA1}$ to the right-continuous process $(\int_{0}^{t}\int_{\mathcal{U}}e^{\beta s}\widetilde{Y}_{s-}\widetilde{Z}_{s}(x)(\mu-\nu)(dx,ds))_{t\in[0,T]}$ and Burkholder-Davis-Gundy
inequalities (\cite{Protter2000} Theorem 48, page 193. For $p=1$), we get
\begin{multline}
\label{mono17}
 E\Bigl[ess\sup_{\tau\in\mathcal{T}_{0,T}}|\int_{0}^{\tau}\int_{\mathcal{U}}e^{\beta s}\widetilde{Y}_{s-}\widetilde{Z}_{s}(x)(\mu-\nu)(dx,ds)|\Bigr]= E\Bigl[\sup_{t\in[0,T]}|\int_{0}^{t}\int_{\mathcal{U}}e^{\beta s}\widetilde{Y}_{s-}\widetilde{Z}_{s}(x)(\mu-\nu)(dx,ds)|\Bigr]\\
\leq cE\Bigl[\sqrt{\int_{0}^{T}\int_{\mathcal{U}}e^{2\beta s}\widetilde{Y}_{s-}^{2}|\widetilde{Z}_{s}(x)-\hat{Z}_{s}|^{2}\nu(dx,ds)+\sum_{0< s\leq T}e^{2\beta s}\widetilde{Y}_{s-}^{2}|\hat{Z}_{s}|^{2}(1-\triangle K_{s})} \Bigr]
\end{multline}
where $c$ is a positive constant (which does not depend
on the choice of the other parameters). The same way  used to obtain
equation $~\eqref{mono14}$ leads to
\begin{multline}
\sqrt{\int_{0}^{T}\int_{\mathcal{U}}e^{2\beta s}\widetilde{Y}_{s-}^{2}|\widetilde{Z}_{s}(x)-\hat{Z}_{s}|^{2}\nu(dx,ds)+\sum_{0< s\leq T}e^{2\beta s}\widetilde{Y}_{s-}^{2}|\hat{Z}_{s}|^{2}(1-\triangle K_{s})}\leq\\
\sqrt{ess\sup_{\tau\in\mathcal{T}_{0,T}}e^{\beta \tau}\widetilde{Y}_{\tau}^{2}} \sqrt{\int_{0}^{T}\int_{\mathcal{U}}e^{\beta s}|\widetilde{Z}_{s}(x)-\hat{Z}_{s}|^{2}\nu(dx,ds)+\sum_{0< s\leq T}e^{\beta s}|\hat{Z}_{s}|^{2}(1-\triangle K_{s})}.
\label{mono18}
\end{multline}
From the  inequalities  $~\eqref{mono17}$, $~\eqref{mono18}$ and
$ab\leq\frac{1}{2}a^{2}+\frac{1}{2}b^{2}$, we have

\begin{multline}
2E\Bigl[ess\sup_{\tau\in\mathcal{T}_{0,T}}|\int_{0}^{\tau}\int_{\mathcal{U}}e^{\beta s}\widetilde{Y}_{s-}\widetilde{Z}_{s}(x)(\mu-\nu)(dx,ds)|\Bigr]\leq
\frac{1}{2}E\Bigl[ess\sup_{\tau\in\mathcal{T}_{0,T}}e^{\beta \tau}\widetilde{Y}_{\tau}^{2}\Bigr]+\\
2c^{2}E\Bigl[\int_{0}^{T}\int_{\mathcal{U}}e^{\beta s}|\widetilde{Z}_{s}(x)-\hat{Z}_{s}|^{2}\nu(dx,ds)+\sum_{0< s\leq T}e^{\beta s}|\hat{Z}_{s}|^{2}(1-\triangle K_{s})\Bigr].
\label{mono19}
\end{multline}
By $~\eqref{mono16}$ and $~\eqref{mono19}$  we  derive that
$$
\frac{1}{2} \||\widetilde{Y}\||_{\mathds{S}^{2,\beta}}^{2}\leq \epsilon^{2}E\Bigl[\int_{0}^{T}e^{\beta s}\widetilde{f}(s)^{2}ds\Bigr]+2c^{2}\|\widetilde{Z}(x)\|_{\mathds{H}^{2,\beta}}^{2}.
$$
This inequality, together with the estimates of $\widetilde{Z}(x)$,
gives
$$
 \||\widetilde{Y}\||_{\mathds{S}^{2,\beta}}^{2}\leq 2\epsilon^{2}(1+2c^{2})E\Bigl[\int_{0}^{T}e^{\beta s}\widetilde{f}(s)^{2}ds\Bigr].
$$
\end{proof}

Before proving the existence and uniqueness results of the solution to RBSDEs associated with parameters $(\xi,f(\omega,t))$, let us recall the following very useful result of the optimal stopping theory, which will be used in our proofs.
\begin{proposition}
\label{proposition1}
Let $(\overline{Y}(S))$ be the family defined for  $S\in\mathcal{T}_{0,T}$ by
\begin{equation}
\overline{Y}(S)=ess\sup_{\tau\in\mathcal{T}_{S,T}}E\Bigl[ \xi_{\tau}+\int_{S}^{\tau}f(u)du/\mathcal{F}_{S}\Bigr],
\end{equation}
\begin{itemize}
  \item There exists a l\`{a}dl\`{a}g optional process $(\overline{Y}_{t})_{t\in[0,T]}$  which aggregates the family $(\overline{Y}(S))$ (i.e.$\overline{Y}_{S}=\overline{Y}(S)$, for all $S\in\mathcal{T}_{0,T}$).
      Moreover, the process  ($\overline{Y}_{t}+\int_{0}^{t}f(u)du)_{t\in[0,T]}$ is a strong supermartingale.
  \item We have $\overline{Y}_{S}=\xi_{S}\vee\overline{Y}_{S+}$ a.s. for all $S\in\mathcal{T}_{0,T}$.
  \item Furthermore, $\overline{Y}_{S+}=ess\sup_{\tau>S}E\Bigl[ \xi_{\tau}+\int_{S}^{\tau}f(u)du/\mathcal{F}_{S}\Bigr]$,  a.s. for all $S\in\mathcal{T}_{0,T}$.
\end{itemize}
\end{proposition}
For the proof of this proposition  the reader is referred to \cite[Proposition A.6]{Ouknine2015}, (See also \cite{Maingueneau1978} when $\xi$ is left- and right-limited, and  \cite[Section B.]{KobylanskiQuenez2012} in the general case).\\
In the following lemma, we prove existence and uniqueness of the
solution to the reflected BSDE from Definition$~\ref{definition1}$ in the
case where the driver $f$ depends only on $s$ and $\omega$, and
we characterize the first component of the solution as the value process of an optimal
stopping problem.
\begin{lemma}
\label{lemme3} Suppose that $E\Bigl[\int_{0}^{T}
f^{2}(t)dt\Bigr]<\infty$. Then, the reflected BSDE from
Definition$~\ref{definition1}$ admits a unique solution
$(Y,Z,A,C)\in\mathds{S}^{2,0}\times\mathds{H}^{2,0}\times\mathds{S}^{2,0}\times\mathds{S}^{2,0}$, and for each $S\in\mathcal{T}_{0,T}$, we have
\begin{equation}
Y_{S}=ess\sup_{\tau\in\mathcal{T}_{S,T}}E\Bigl[\xi_{\tau}+\int_{S}^{\tau}
f(t)dt|\mathcal{F}_{S} \Bigr] \,\,\,\ a.s.
\label{mono20}
\end{equation}
Moreover, $Y_{S}=\xi_{S}\vee Y_{S+}$ a.s.
\end{lemma}

\begin{proof}
See Appendix.

\end{proof}


\subsection{The case of a general Lipschitz driver.}

Let is now assume that $f$ is a general Lipschitz driver, that is $f:=f(s,y,z)$. In the following theorem, we prove existence and uniqueness of the solution to the reflected BSDE from Definition$~\ref{definition1}$ by using the Banach  fixed-point theorem and  Remark$~\ref{remarque2}$.
\begin{remark}
\label{remarque2}
 Let $\beta>0$. For $\varphi\in\mathds{S}^{2,0}$, we have
 \begin{equation}
E\Bigl[\int_{0}^{T}e^{\beta t}|\varphi_{t}|^{2}dt \Bigr]\leq TE\Bigl[ess\sup_{\tau\in\mathcal{T}_{0,T}}e^{\beta \tau}|\varphi_{\tau}|^{2}\Bigr].
\end{equation}
\end{remark}
Indeed, by applying Fubini's theorem, we get
$$E\Bigl[\int_{0}^{T}e^{\beta t}|\varphi_{t}|^{2}dt\Bigr]= \int_{0}^{T}E[e^{\beta
t}|\varphi_{t}|^{2}]dt\leq
\int_{0}^{T}E\Bigl[ess\sup_{\tau\in\mathcal{T}_{0,T}}e^{\beta \tau}|\varphi_{\tau}|^{2}\Bigr]ds=
TE\Bigl[ess\sup_{\tau\in\mathcal{T}_{0,T}}e^{\beta\tau}|\varphi_{\tau}|^{2}\Bigr]$$

\begin{theorem}
Let $\xi$ be a left-limited and r.u.s.c. process in $\mathds{S}^{2,0}$
and let $f$ be a Lipschitz driver. The reflected BSDE with parameters
$(f,\xi)$ from Definition$~\ref{definition1}$  admits a unique
solution $(Y,Z,A,C)\in\mathds{S}^{2,0}\times\mathds{H}^{2,0}\times\mathds{S}^{2,0}\times\mathds{S}^{2,0}$.\\
Furthermore, if $(\xi_t)$ is assumed l.u.s.c. along stopping times, then  $(A_t)$ is continuous.
\end{theorem}

\begin{proof}
For each $\beta>0$ and  since $(\mathds{S}^{2,\beta},\||.|\|_{\mathds{S}^{2,\beta}})$ and $(\mathds{H}^{2,\beta},\|.\|_{\mathds{H}^{2,\beta}})$ are Banach spaces (by Remark$~\ref{Banach}$) it follows that $(\mathds{S}^{2,\beta}\times\mathds{H}^{2,\beta}, \|(.,.)\|_{2,\beta})$ is a Banach space with:
$$\|(Y,Z(.))\|^{2}_{2,\beta}=\||Y|\|^{2}_{\mathds{S}^{2,\beta}}+\|Z(.)\|^{2}_{\mathds{H}^{2,\beta}}$$
 for all $(Y,Z)\in\mathds{S}^{2,\beta}\times\mathds{H}^{2,\beta}$. After, we define an application
$\Phi$ from $\mathds{S}^{2,\beta}\times\mathds{H}^{2,\beta}$  into itself as follows: for a given $(y,z)\in\mathds{S}^{2,\beta}\times\mathds{H}^{2,\beta}$, we
set $(Y,Z)=\Phi(y,z)$ where $(Y,Z)$ the first
two  components of the solution to the reflected BSDE associated with driver
$f:=f(t,y_t,z_t)$ and with obstacle $\xi_t$. Let $(A,C)$
be the associated Mertens process, constructed as in
lemma$~\ref{lemme3}$. The mapping $\Phi$ is well-defined by
Lemma$~\ref{lemme3}$.

Let $(y,z)$ and $(y',z')$ be two elements of
$\mathds{S}^{2,\beta}\times\mathds{H}^{2,\beta}$. We set $(Y,Z)=\Phi(y,z)$
 and  $(Y',Z')=\Phi(y',z')$. We also set $\widetilde{Y}=Y-Y'$, $\widetilde{Z}=Z-Z'$,
$\widetilde{y}=y-y'$ and  $\widetilde{z}=z-z'$.

Let us prove that for a suitable choice of the parameter $\beta>0$, the mapping
$\Phi$ is a contraction from the Banach space
$\mathds{S}^{2,\beta}\times\mathds{H}^{2,\beta}$ into itself. Indeed, By applying
Lemma$~\ref{lemme2}$, we have, for all $\epsilon>0$ and for all
$\beta\geq\frac{1}{\epsilon^{2}}$:
\begin{eqnarray*}
\||\widetilde{Y}|\|_{\mathds{S}^{2,\beta}}^{2}+\|\widetilde{Z}\|^{2}_{\mathds{H}^{2,\beta}}\leq
\epsilon^{2}(3+4c^{2})E\Bigl[\int_{0}^{T}e^{\beta s}(f(s,y,z)-f(s,y',z'))^{2}ds\Bigr].
\end{eqnarray*}
By using the Lipschitz property of $f$ and the fact that
$(a+b)^2\leq2a^2+2b^2$, for all $(a,b)\in\R^{2}$, we obtain
\begin{multline}
E\Bigl[\int_{0}^{T}e^{\beta s}(f(s,y,z)-f(s,y',z'))^{2}ds\Bigr]\leq C_{L}(E\Bigl[\int_{0}^{T}e^{\beta s}|\widetilde{y}|^{2}ds \Bigr]+ \\
E\Bigl[\int_{0}^{T}e^{\beta s} \int_{\mathcal{U}}|\widetilde{z}_{s}(x)-\hat{z}_{s}|^{2}\nu(dx,ds)+\sum_{0< s\leq T}e^{\beta s}|\hat{z}_{s}|^{2}(1-\triangle K_{s})\Bigr])
\end{multline}
where $\hat{z}_{t}=\int_{\mathcal{U}}\widetilde{z}_{t}(x)\nu(\{t\},dx)$, and  $C_{L}$ is a positive constant depending on the Lipschitz constant
$L$ only. Thus, for all $\epsilon>0$ and for all
$\beta\geq\frac{1}{\epsilon^{2}}$ we have:
$$\||\widetilde{Y}|\|_{\mathds{S}^{2,\beta}}^{2}+\|\widetilde{Z}\|^{2}_{\mathds{H}^{2,\beta}}\leq
    \epsilon^{2}C_{L}(3+4c^{2})\Bigl(\|\widetilde{y}\|_{\mathds{S}^{2,\beta}}^{2}+\|\widetilde{z}\|_{\mathds{H}^{2,\beta}}^{2}\Bigr).$$
The previous inequality, combined with in Remark$~\ref{remarque2}$, gives
$$\||\widetilde{Y}|\|_{\mathds{S}^{2,\beta}}^{2}+\|\widetilde{Z}\|^{2}_{\mathds{H}^{2,\beta}}\leq
    \epsilon^{2}C_{L}(3+4c^{2})(T+1)\Bigl(\||\widetilde{y}|\|_{\mathds{S}^{2,\beta}}^{2}+\|\widetilde{z}\|_{\mathds{H}^{2,\beta}}^{2}\Bigr)$$
Thus, for $\epsilon>0$ such that
$\epsilon^{2}C_{L}(3+4c^{2})(T+1)<1$ and $\beta\geq0$ such that
$\beta\geq\frac{1}{\epsilon^{2}}$, the mapping $\Phi$ is a
contraction. By the Banach fixed-point theorem, we get that $\Phi$
has a unique fixed point in $\mathds{S}^{2,\beta}\times\mathds{H}^{2,\beta}$. We thus have
the existence and uniqueness of the solution to the reflected BSDE.
\end{proof}
\begin{remark}
Let $\mu$ be an integer-valued random measure on $\R_{+}\times\mathcal{U}$ not
necessarily discrete. Then $\nu$ the predictable compensator of the measure $\mu$, relative to the filtration $(\mathcal{F}_{t})_{(t\geq 0)}$ can be disintegrated as follows
$$
\nu(\omega,dt,dx)=dK_{t}(\omega)\phi_{\omega,t}(dx),
$$
where $K$ is a right-continuous nondecreasing predictable process such that $K_{0}=0$, but
$\phi$  is in general only a transition measure (instead of transition probability)  from $(\Omega\times[0,T],\mathcal{P})$ into $(\mathcal{U},\mathcal{E})$. Notice that when $\mu$ is discrete one can choose $\phi$ to be a transition probability, therefore  $\phi(\mathcal{U})$ = 1 and $\nu(\{t\}\times \mathcal{U})=\Delta K_{t}$ (a property used in the
previous sections). When $\mu$ is not discrete, let us suppose that $\nu^{d}$ (the discontinuous part of $\nu$) can be disintegrated as follows
\begin{equation}
\nu^{d}(\omega,dt,dx)=\Delta K_{t}(\omega)\phi^{d}_{\omega,t}(dx),\,\,\,\,\ \phi^{d}_{\omega,t}(\mathcal{U})=1
 \label{nu}
\end{equation}
where $\phi^{d}$  is a transition probability from $(\Omega\times[0,T],\mathcal{P})$ into $(\mathcal{U},\mathcal{E})$. In particular $\nu^{d}(\{t\}\times \mathcal{U})=\Delta K_{t}$. Then, when $~\eqref{nu}$ and a martingale representation theorem for $\mu$ hold, all
the results of this paper are still valid and can be proved proceeding along the same
lines.
\end{remark}
\subsection{Comparison theorem for reflected BSDEs.}

In this section we will  provide a comparison result between solutions of  reflected BSDEs
which is an important  theorem. So assume there exists two solutions $(Y^i, Z^i, A^i, C^i)\,\,\,\,\ i=1,2$ of $~\eqref{moneq4}$  associated with $(\xi^i, f_{i})\,\,\,\,\ i=1,2$. Next, we show that $(Y^{1}_{t}-Y^{2}_{t})^{+}\in\mathds{S}^{2,0}$ where $(Y^{1}-Y^{2})^{+}=sup(Y^{1}-Y^{2},0)$.

\begin{theorem}
\label{comparison1}
Let $(Y^i, Z^i, A^i, C^i)$ be a solution of reflected BSDE with parameters $(\xi^i, f_i)$, $i=1,2$. Assume that $\xi^1\leq \xi^2$, for a.s. $t\in[0,T]$ and either
\begin{equation}
f_{1}\,\,\,\ be \,\,\,\ Lipschitz \,\,\,\ and \,\,\,\ 1_{\{Y^{1}_{t-}>Y^{2}_{t-}\}}(f_{1}(t,Y^{1}_{t},Z^{1}_{t})- f_{2}(t, Y^{1}_{t},Z^{1}_{t}))\leq 0
\end{equation}
for a.s. $t\in[0,T]$ or
\begin{equation}
f_{2}\,\,\,\ be \,\,\,\ Lipschitz \,\,\,\ and \,\,\,\ \ 1_{\{Y^{1}_{t-}>Y^{2}_{t-}\}}(f_{1}(t,Y^{2}_{t},Z^{2}_{t})- f_{2}(t, Y^{2}_{t},Z^{2}_{t}))\leq 0
\end{equation}
for a.s.  $t\in[0, T]$. Then $$Y^{ 1}_{ t} \leq Y^{ 2}_{ t }, \,\,\,\ for \,\,\,\ all \,\,\,\ t\in[0,T]\,\,\,\ a.s.$$
\end{theorem}

\begin{proof}
Set $Z=Z^{1}-Z^{2}$, $A=A^{1}-A^{2}$ and $C=C^{1}-C^{2}$. By  $~\ref{moneq4}$, for $\tau,\sigma\in\mathcal{T}_{0,T}$ such that $\tau \leq \sigma$ we have
$$
Y^{1}_{\tau}-Y^{2}_{\tau}=Y^{1}_{\sigma}-Y^{2}_{\sigma} +\int_{\tau}^{\sigma}(f_{1}(s,Y^{1}_{s},Z^{1}_{s}(.))-f_{2}(s,Y^{2}_{s},Z^{2}_{s}(.)))ds-\int_{\tau}^{\sigma}\int_{\mathcal{U}}Z_{s}(x)(\mu-\nu)(dx,ds)
+A_{\sigma}-A_{\tau}+C_{\sigma-}-C_{\tau-}
$$

By Corollary$~\ref{corollaire2}$,  we obtain

\begin{multline}
  ((Y^{1}_{\tau}-Y^{2}_{\tau})^{+})^{2}+\int_{\tau}^{\sigma}\int_{\mathcal{U}}1_{\{Y^{1}_{s}>Y^{2}_{s}\}}|Z_{s}(x)-\hat{Z}_{s}|^{2}\nu(dx,ds)+\sum_{\tau< s\leq\sigma}1_{\{Y^{1}_{s}>Y^{2}_{s}\}}|\hat{Z}_{s}|^{2}(1-\triangle K_{s})\\
  \leq  ((Y^{1}_{\sigma}-Y^{2}_{\sigma})^{+})^{2}+ 2\int_{\tau}^{\sigma}(Y^{1}_{s-}-Y^{2}_{s-})^{+}(f_{1}(s,Y^{1}_{s},Z^{1}_{s})- f_{2}(s, Y^{2}_{s},Z^{2}_{s}))ds\\
  + 2\int_{\tau}^{\sigma}(Y^{1}_{s-}-Y^{2}_{s-})^{+}dA_{s}-
  2\int_{\tau}^{\sigma}\int_{\mathcal{U}}(Y^{1}_{s-}-Y^{2}_{s-})^{+}Z_{s}(x)(\mu-\nu)(dx,ds)
  -2\sum_{\tau\leq s<\sigma}(Y^{1}_{s}-Y^{2}_{s})^{+}\Delta C_{s}
  \label{comparison2}
\end{multline}
where $\Delta C_{s}=C_{s}-C_{s-}=-(Y_{s+}-Y_{s})$ and $\hat{Z}_{t}=\int_{\mathcal{U}}Z_{t}(x)\nu(\{t\},dx)$. By using the theorem assumptions and  Lipschitz condition of driver $f$ and $2ab\leq(2a)^{2}+(\frac{b}{2})^{2}$, we have

\begin{multline}
2(Y^{1}_{s-}-Y^{2}_{s-})^{+}(f_{1}(s,Y^{1}_{s},Z^{1}_{s})-f_{2}(s, Y^{2}_{s},Z^{2}_{s}))ds \leq 2(Y^{1}_{s-}-Y^{2}_{s-})^{+}(f_{1}(s,Y^{1}_{s},Z^{1}_{s})- f_{1}(s, Y^{2}_{s},Z^{2}_{s}))ds\\
\leq 2L((Y^{1}_{s}-Y^{2}_{s})^{+})^{2} +
2L(Y^{1}_{s}-Y^{2}_{s})^{+}\Bigl(\int_{\mathcal{U}}|Z_{s}(x)-\hat{Z}_{s}|^{2}\nu(ds,dx)
+|\hat{Z}_{s}|^{2}(1-\triangle K_{t})\Bigr)^{\frac{1}{2}}\\
\leq (2L+4L^{2})((Y^{1}_{s}-Y^{2}_{s})^{+})^{2}+
\frac{1}{4}\Bigl(1_{\{Y^{1}_{s}>Y^{2}_{s}\}}\int_{\mathcal{U}}|Z_{s}(x)-\hat{Z}_{s}|^{2}\nu(ds,dx)+1_{\{Y^{1}_{s}>Y^{2}_{s}\}}|\hat{Z}_{s}|^{2}(1-\triangle K_{t})\Bigr)
  \label{comparison3}
\end{multline}

Let's go back to $~\eqref{comparison2}$, its clear that $-2\sum_{\tau\leq s<\sigma}(Y^{1}_{s}-Y^{2}_{s})^{+}\Delta C_{s}\leq0$ and observe that in $\{Y^{1}_{s-}>Y^{2}_{s-}\}$ we have
$Y^{1}_{s-}>\xi^{1}_{s-}$ for a.e. $s\in[0,T]$ then
$$\int_{\tau}^{\sigma}(Y^{1}_{s-}-Y^{2}_{s-})^{+}dA_{s}=\int_{\tau}^{\sigma}1_{\{Y^{1}_{s-}>Y^{2}_{s-}\}}(Y^{1}_{s-}-Y^{2}_{s-})^{+}\mid
Y^{1}_{s-}-\xi^{1}_{s-}\mid^{-1}(Y^{1}_{s-}-\xi^{1}_{s-})dA_{s}$$
Hence by condition $~\eqref{moneq6}$
$\int_{\tau}^{\sigma}(Y^{1}_{s-}-Y^{2}_{s-})^{+}dA_{s}\leq 0$, $t\in[0,\sigma]$  of the Definition$~\ref{definition1}$. From this, $~\ref{comparison2}$ and $~\ref{comparison3}$  it follows that

\begin{eqnarray*}
  ((Y^{1}_{\tau}-Y^{2}_{\tau})^{+})^{2}&+& \frac{3}{4}\Bigl(\int_{\tau}^{\sigma}\int_{\mathcal{U}}1_{\{Y^{1}_{s}>Y^{2}_{s}\}}|Z_{s}(x)-\hat{Z}_{s}|^{2}\nu(dx,ds)+\sum_{\tau< s\leq\sigma}1_{\{Y^{1}_{s}>Y^{2}_{s}\}}|\hat{Z}_{s}|^{2}(1-\triangle K_{s})\Bigr)\\
  &\leq&  ((Y^{1}_{\sigma}-Y^{2}_{\sigma})^{+})^{2}+ (2L+4L^{2})\int_{\tau}^{\sigma}((Y^{1}_{s}-Y^{2}_{s})^{+})^{2}ds\\
&-& 2\int_{\tau}^{\sigma}\int_{\mathcal{U}}(Y^{1}_{s-}-Y^{2}_{s-})^{+}Z_{s}(x)(\mu-\nu)(dx,ds)
\end{eqnarray*}

which implies that
$$
  ((Y^{1}_{\tau}-Y^{2}_{\tau})^{+})^{2}\leq ((Y^{1}_{\sigma}-Y^{2}_{\sigma})^{+})^{2}+ (2L+4L^{2})\int_{\tau}^{\sigma}((Y^{1}_{s}-Y^{2}_{s})^{+})^{2}ds-(M_{\sigma}-M_{\tau})
$$
Let $\{\sigma_{n}\}\in\mathcal{T}_{0,T}$ be a fundamental sequence for the local martingale $M$. Changing $\sigma_{n}$ with $\sigma$ in the above inequality, taking expected value
$$
  E((Y^{1}_{\tau}-Y^{2}_{\tau})^{+})^{2}\leq E((Y^{1}_{\sigma_{n}}-Y^{2}_{\sigma_{n}})^{+})^{2}+ (2L+4L^{2})E\int_{\tau}^{\sigma_{n}}((Y^{1}_{s}-Y^{2}_{s})^{+})^{2}ds
$$
 and passing to the limit with $n\longrightarrow \infty$ gives
$$
  E((Y^{1}_{\tau}-Y^{2}_{\tau})^{+})^{2}\leq(2L+4L^{2})E\int_{\tau}^{T}((Y^{1}_{s}-Y^{2}_{s})^{+})^{2}ds
$$
so applying the Gronwall’s lemma we get $E((Y^{1}_{\tau}-Y^{2}_{\tau})^{+})^{2}=0$. By Proposition$~\ref{AshkanNikeghbali}$ (The Section Theorem), $(Y^{1}_{t}-Y^{2}_{t})^{+}=0$, $t\in[0,T]$. And  the desired result yields.
\end{proof}

\begin{remark}
Let $f_1$, $f_2$, $\xi^{1}$, $\xi^{2}$ be satisfying the same
assumptions as in Theorem$~\ref{comparison1}$. Then $(Y^{1}_{t}-Y^{2}_{t})^{+}\in\mathds{S}^{2,0}$.

\end{remark}
\begin{proof}
Set $Z=Z^{1}-Z^{2}$, $A=A^{1}-A^{2}$ and $C=C^{1}-C^{2}$.  By the assumptions and the Tanaka formula for optional
semimartingales (see \cite[Section 3, page 538]{Lenglart1980}) applied to $(Y^{1}_{t}-Y^{2}_{t})^{+}$  (see also \cite[Theorem 4.1]{baadi2018} for the a general filtration), for every $\tau\in\mathcal{T}_{0,T}$ we have
\begin{eqnarray*}
  (Y^{1}_{\tau}-Y^{2}_{\tau})^{+} &\leq&
  (\xi^{1}-\xi^{2})^{+}+\int_{\tau}^{T}1_{\{Y^{1}_{s-}>Y^{2}_{s-}\}}(f^{1}(s,Y^{1},Z^{1})-f^{2}(s,Y^{2},Z^{2}))ds \\ &+& \int_{\tau}^{\tau}1_{\{Y^{1}_{s-}>Y^{2}_{s-}\}}dA_{s} +\int_{\tau}^{T}1_{\{Y^{1}_{s-}>Y^{2}_{s-}\}}dC_{s-}\\ &-& \int_{\tau}^{T}\int_{\mathcal{U}}1_{\{Y^{1}_{s-}>Y^{2}_{s-}\}}Z_{s}(x)(\mu-\nu)(dx,ds).
\end{eqnarray*}
By the same way in \ref{comparison1}, we have

\begin{eqnarray*}
  (Y^{1}_{\tau}-Y^{2}_{\tau})^{+} &\leq &
L\int_{\tau}^{T}1_{\{Y^{1}_{s-}>Y^{2}_{s-}\}}(Y^{1}_{s}-Y^{2}_{s})ds-\int_{\tau}^{T}\int_{\mathcal{U}}1_{\{Y^{1}_{s-}>Y^{2}_{s-}\}}Z_{s}(x)(\mu-\nu)(dx,ds)\\
&+& L\int_{\tau}^{T}1_{\{Y^{1}_{s-}>Y^{2}_{s-}\}}\Bigl(\int_{\mathcal{U}}
|Z_{s}(x)-\hat{Z}_{s}|^{2}\nu(ds,dx)
+ |\hat{Z}_{s}|^{2}(1-\triangle K_{s})\Bigr)^{\frac{1}{2}}
\end{eqnarray*}
By taking conditional expectation we get
\begin{eqnarray*}
  (Y^{1}_{\tau}-Y^{2}_{\tau})^{+} \leq
LE\Bigl(\int_{\tau}^{T}1_{\{Y^{1}_{s-}>Y^{2}_{s-}\}}(Y^{1}_{s}-Y^{2}_{s})ds &+&
\int_{\tau}^{T}1_{\{Y^{1}_{s-}>Y^{2}_{s-}\}}\Bigl(\int_{\mathcal{U}}
|Z_{s}(x)-\hat{Z}_{s}|^{2}\nu(ds,dx)\\
&+& |\hat{Z}_{s}|^{2}(1-\triangle K_{s})\Bigr)^{\frac{1}{2}}|\mathcal{F}_{\tau}\Bigr)
\end{eqnarray*}
Hence
$$
((Y^{1}_{\tau}-Y^{2}_{\tau})^{+})^{2} \leq
C_{T,L}E\Bigl(\int_{0}^{T}(Y^{1}_{s}-Y^{2}_{s})^{2}ds+
\int_{0}^{T}\int_{\mathcal{U}}|Z_{s}(x)-\hat{Z}_{s}|^{2}\nu(ds,dx)
+\sum_{0\leq s \leq T}|\hat{Z}_{s}|^{2}(1-\triangle K_{s})|\mathcal{F}_{\tau}\Bigr)
$$
where $C_{T,L}$ is a positive constant depending only on the Lipschitz constant $L$ and $T$.  By Doob’s inequality, taking the essential supremum over $\tau\in\mathcal{T}_{0,T}$, and expectation  we get
\begin{eqnarray*}
  E\Bigl(ess\sup_{\tau\in\mathcal{T}_{0,T}}((Y^{1}_{\tau}-Y^{2}_{\tau})^{+})^{2}\Bigr)&\leq&
C_{T,L}E\Bigl(\int_{0}^{T}(Y^{1}_{s}-Y^{2}_{s})^{2}ds\Bigr)\\ &+&
C_{T,L}E\Bigl(\int_{0}^{T}\int_{\mathcal{U}}|Z_{s}(x)-\hat{Z}_{s}|^{2}\nu(ds,dx)
+\sum_{0\leq s \leq T}|\hat{Z}_{s}|^{2}(1-\triangle K_{s})\Bigr)
\end{eqnarray*}
and by Remark$~\ref{remarque2}$  we have $E\Bigl(\int_{0}^{T}(Y^{1}_{s}-Y^{2}_{s})^{2}ds\Bigr)\leq TE\Bigl(ess\sup_{\tau\in\mathcal{T}_{0,T}}(Y^{1}_{\tau}-Y^{2}_{\tau})^{2}\Bigr)$ then

$$
  E\Bigl(ess\sup_{\tau\in\mathcal{T}_{0,T}}((Y^{1}_{\tau}-Y^{2}_{\tau})^{+})^{2}\Bigr)\leq
C_{T,L}\||Y^{1}-Y^{2}\||_{\mathds{S}^{2,0}}^{2}+
C_{T,L}\|Z\|_{\mathds{H}^{2,0}}^{2}<\infty
$$
Hence, the desired result yields : $(Y^{1}_{t}-Y^{2}_{t})^{+}\in\mathds{S}^{2,0}$
\end{proof}

\section{Connection to the Optimal Stopping Problem}\label{sec3}

In this section, we provide some links between the reflected BSDE studied in the first part and an optimal stopping problem with running gains $f$ and non stopping reward $\xi$ . We are thus interested in the following optimal stopping problem
\begin{equation}
V(S)=ess\sup_{\tau\in\mathcal{T}_{S,T}}E\Bigl[\xi_{\tau}+\int_{S}^{\tau}
f(t,Y_{t},Z_{t}(.))dt|\mathcal{F}_{S} \Bigr] \,\,\,\ a.s.
\label{mono23}
\end{equation}
By Proposition$~\ref{proposition1}$ there exists a l\`{a}dl\`{a}g optional process $(V_{t})_{t\in[0,T]}$ which aggregates the family $(V(S))$ (i.e.$V_{S}=V(S)$ a.s., for all $S\in\mathcal{T}_{0,T}$).
Moreover, the process $(V_{t})_{t\in[0,T]}$ is the smallest strong supermartingale greater than or
equal to $(\xi_{t}+\int_{0}^{t}f(t,Y_{s},Z_{s}(.))ds)_{t\in[0,T]}$.\\
We show how what precedes can be put in use for proving that the solution $Y$ to the reflected BSDE from Definition$~\ref{definition1}$ with parameters the $(f, \xi)$ is also the value function of the  optimal stopping problem  $~\eqref{mono23}$ (i.e.  $Y_{S}=V_{S}$ a.s. for all $S\in\mathcal{T}_{0,T}$). We also investigate the question of the existence of an optimal stopping time under an additional  assumptions on the process $\xi$.

\begin{proposition}
\label{optimalprop1}
Let $f$ be a predictable Lipschitz driver. Let
$(\xi_{t}, 0 \leq t\leq T)$ be a left-limited r.u.s.c. process in $\mathds{S}^{2,0}$
\begin{enumerate}
  \item The solution to the reflected BSDE from Definition$~\ref{definition1}$ is a solution to the optimal stopping problem $$Y_{S}=ess\sup_{\tau\in\mathcal{T}_{S,T}}E\Bigl[\xi_{\tau}+\int_{S}^{\tau}
f(t,Y_{t},Z_{t}(.))dt|\mathcal{F}_{S} \Bigr]$$
  \item For each $S\in\mathcal{T}_{0,T}$ and each $\epsilon \geq0$, the stopping time $\tau_{S}^{\epsilon}=inf\{t\geq S: Y_{t}\leq\xi_{t}+\epsilon\}$  is $(\epsilon)$-optimal for problem $~\eqref{mono23}$, that is: $$Y_{S}\leq E\Bigl[\xi_{\tau_{S}^{\epsilon}}+\int_{S}^{\tau_{S}^{\epsilon}}f(t,Y_{t},Z_{t}(.))dt|\mathcal{F}_{S}\Bigr]+\epsilon$$
  \item If in addition $\xi$ is left u.s.c. process, then the stopping time $\tau_{S}^{*}=inf\{t\geq S: Y_{t}=\xi_{t}\}\wedge T$ is optimal and is the smallest of all optimal stopping times.
\end{enumerate}
\end{proposition}
\begin{remark}
The proof of  the same result in the case of  $f$-conditional expectations (non linear conditional expectation) and Brownian–Poisson filtration is given in \cite[Theorem 4.2 and  Proposition 4.3]{Ouknine2015} ( see also \cite[Proposition 3.1]{foresta2019} for the càdlàg case).
\end{remark}

The next lemma will be used in the proof of this proposition which is given in \cite[Lemma 4.1]{Ouknine2015}.
\begin{lemma}
\label{lemmaprop2}
Let $\xi$ be a left-limited r.u.s.c. process in $\mathds{S}^{2,0}$ and  $f$ be a predictable Lipschitz driver. Let $(Y,Z,A,C)$ be the solution to the reflected BSDE with parameters $(f,\xi)$ as in Definition$~\ref{definition1}$. Then for a.e. $\omega\in\Omega$,
\begin{enumerate}
  \item $Y_{\tau_{S}^{\epsilon}}\leq \xi_{\tau_{S}^{\epsilon}}+\epsilon$
  \item The map $t\longmapsto A_{t}(\omega)+C_{t-}(\omega)$ is constant on $[S(\omega), \tau_{S}^{\epsilon}(\omega)]$.
\end{enumerate}
\end{lemma}

\begin{proof}[Proof of Proposition \ref{optimalprop1}]
Let $S\in\mathcal{T}_{0,T}$, let $\tau\in\mathcal{T}_{S,T}$ and consider the reflected BSDE $~\eqref{moneq4}$ from Definition$~\ref{definition1}$  between $S$ and $\tau$:
$$
Y_{S} =
Y_{\tau}+\int_{S}^{\tau}f(t,Y_{t},Z_{t}(.))dt-\int_{S}^{\tau}\int_{\mathcal{U}}Z_{t}(x)(\mu-\nu)(dx,dt)
+A_{\tau}-A_{S}+C_{\tau-}-C_{S-}
$$
By taking the conditional expectation on both sides of the above inequality, we obtain
\begin{equation}
Y_{S} =  E\Bigl[Y_{\tau}+\int_{S}^{\tau}f(t,Y_{t},Z_{t}(.))dt+A_{\tau}-A_{S}+C_{\tau-}-C_{S-}|\mathcal{F}_{S} \Bigr]
 \label{optimaleq1}
\end{equation}
since  $t\longmapsto A_{t}(\omega)+C_{t-}(\omega)$ is increasing and $Y_{t}\geq \xi_{t}$, we get
\begin{equation}
Y_{S}\geq E\Bigl[\xi_{\tau}+\int_{S}^{\tau}f(t,Y_{t},Z_{t}(.))dt|\mathcal{F}_{S} \Bigr]
 \label{optimaleq2}
\end{equation}
To prove the reverse inequality, consider $\tau_{S}^{\epsilon}$. It holds by \cite[Lemma 4.1]{Ouknine2015} that $Y_{\tau_{S}^{\epsilon}}\leq \xi_{\tau_{S}^{\epsilon}}+\epsilon$ on $\{\tau_{S}^{\epsilon}<T\}$. Then, by Lemma$~\ref{lemmaprop2}$, for all $S(\omega)\leq t\leq\tau_{S}^{\epsilon}(\omega)$ we have   $$A_{\tau_{S}^{\epsilon}}+C_{\tau_{S}^{\epsilon}-}=A_{t}-C_{t-}$$
Considering all this in $~\eqref{optimaleq1}$ we have

\begin{equation}
 Y_{S}= E\Bigl[Y_{\tau_{S}^{\epsilon}}+\int_{S}^{\tau_{S}^{\epsilon}}f(t,Y_{t},Z_{t}(.))dt|\mathcal{F}_{S} \Bigr]\leq E\Bigl[\xi_{\tau_{S}^{\epsilon}}+\int_{S}^{\tau_{S}^{\epsilon}}f(t,Y_{t},Z_{t}(.))dt|\mathcal{F}_{S} \Bigr]+\epsilon
  \label{optimaleq3}
\end{equation}
As $\epsilon$ is an arbitrary positive number, we get the first and second points. Let us now prove the tired statement. notice $\tau_{S}^{\epsilon}$ that  are non increasing in $\epsilon$ and that $\tau_{S}^{\epsilon}\leq\tau_{S}^{*}$, thus $\tau_{S}^{\epsilon}\longmapsto\tau_{S}^{0}$ when $\epsilon\longrightarrow 0$. Since $\xi_{t}$ is left u.s.c.e., we have from $~\eqref{optimaleq3}$
$$
 E[Y_{S}]\leq lim\sup_{\epsilon\longrightarrow 0} E\Bigl[\xi_{\tau_{S}^{\epsilon}}+\int_{S}^{\tau_{S}^{\epsilon}}f(t,Y_{t},Z_{t}(.))dt \Bigr]\leq E\Bigl[\xi_{\tau_{S}^{0}}+\int_{S}^{\tau_{S}^{0}}f(t,Y_{t},Z_{t}(.))dt\Bigr]
$$
Thus we have
$$
 E[Y_{S}]=E\Bigl[\xi_{\tau_{S}^{0}}+\int_{S}^{\tau_{S}^{0}}f(t,Y_{t},Z_{t}(.))dt\Bigr]
$$
then $\tau_{S}^{0}$ is optimal and $\tau_{S}^{0}\leq\tau_{S}^{*}$. Since $\tau_{S}^{0}$ is optimal it holds that $Y_{\tau_{S}^{0}}=\xi_{\tau_{S}^{0}}$ (see  \cite[Theorem 2.12. page 111]{ElKaroui1981} ), and thus by the definition of $\tau_{S}^{*}$, $\tau_{S}^{*}\leq\tau_{S}^{0}$. This  proves that $\tau_{S}^{*}=\tau_{S}^{0}$ and proves also that $\tau_{S}^{*}$ is the smallest optimal stopping time.
\end{proof}

\section{Appendix} \label{App:AppendixA}

We start this section by the following observation that can be found
in \cite[Remark A.14]{Ouknine2015}.
\begin{proposition}
\label{propA1} Let $Y$ be a cadlag process. Then,
$\sup_{t\in[0,T]}Y_{t}$ is a random variable and we have
$$\sup_{t\in[0,T]}Y_{t}= ess\sup_{t\in[0,T]}Y_{t}=ess\sup
_{\tau\in\mathcal{T}_{0,T}}Y_{t}.$$
\end{proposition}
The following definition can be found in \cite[Appendix 1]{DellacherieMeyer1980}.
\begin{definition}
\label{defnA1}
 Let $(Y)_{t\in[0,T]}$  be an optional process. We say that $Y$ is a strong (optional) supermartingale
if \begin{itemize}
    \item $Y_{\tau}$ is integrable for all $\tau\in\mathcal{T}_{0,T}$ and
    \item $Y_{S}\geq E[Y_{\tau}|\mathcal{F}_{S}]$ a.s., for all
    $S,\tau\in\mathcal{T}_{0,T}$ such that $S\leq \tau $ a.s.
\end{itemize}
\end{definition}

We recall the change of variables formula for
optional semimartingales which are not necessarily right continuous. The result
can be seen as a generalization of the classical It\^{o} formula and can
be found in \cite[Theorem 8.2]{Galchouk1981}
(see also \cite[Section 3, page 538]{Lenglart1980}). We recall the result
in our framework in which the underlying filtered probability space
satisfies the usual conditions.
\begin{theorem}
\label{galchouklenglart} \textbf{(Gal'chouk-Lenglart)}  Let
$n\in\N$. Let $X$ be an n-dimensional optional semimartingale, i.e.
$X = (X_1,... ,X_n)$ is an n-dimensional optional process with
decomposition $X_{t}^{k}=X_{0}^{k}+M_{t}^{k}+A_{t}^{k}+B_{t}^{k}$,
for all $k\in\{1,...,n\}$ where $M_{t}^{k}$ is a (c\`{a}dl\`{a}g) local
martingale, $A_{t}^{k}$ is a right-continuous process of finite
variation such that $A_{0}=0$ and $B_{t}^{k}$ is a left-continuous
process of finite variation which is purely discontinuous and such
that $B_{0-}=0$. Let $F$ be a twice continuously differentiable
function on $\R^{n}$. Then, almost surely, for all $t\geq0$,
\begin{eqnarray*}
  F(X_{t}) &=& F(X_{0})+\sum_{k=1}^{n}\int_{]0,t]}D^{k}F(X_{s-})d(M^{k}+A^{k})_{s}
           + \frac{1}{2}\sum_{k,l=1}^{n}\int_{]0,t]}D^{k}D^{l}F(X_{s-})d<M^{kc},M^{lc}>_{s}\\
           &+& \sum_{0<s\leq t}\Bigl[F(X_{s})-F(X_{s-})-\sum_{k=1}^{n}D^{k}F(X_{s-})\Delta X^{k}_{s}\Bigr]
           + \sum_{k=1}^{n}\int_{[0,t[}D^{k}F(X_{s})d(B^{k})_{s+} \\
           &+& \sum_{0\leq s<t}\Bigl[F(X_{s+})-F(X_{s})-\sum_{k=1}^{n}D^{k}F(X_{s})\Delta_{+}X^{k}_{s}\Bigr]
\end{eqnarray*}
where $D^k$ denotes the differentiation operator with respect to the
$k$-th coordinate, and $M^{kc}$ denotes the continuous part of
$M^k$.
\end{theorem}
\begin{corollary}
\label{corollaire1} Let $Y$ be a one-dimensional optional
semimartingale with decomposition $Y_t=Y_0+M_t+A_t+B_t$, where $M$,
$A$ and $B$ are as in the above theorem. Let $\beta
>0$. Then, almost surely, for all $t$ in $[0,T]$,
\begin{eqnarray*}
  e^{\beta t}Y_{t}^{2} &=&  Y_{0}^{2}+\int_{0}^{t}\beta e^{\beta s}Y_{s}^{2}ds+2\int_{0}^{t}e^{\beta s}Y_{s-}d(A+M)_{s}
                       + \int_{0}^{t} e^{\beta s}d<M^{c},M^{c}>_{s}\\
                       &+& \sum_{0<s\leq t}e^{\beta s}(Y_{s}-Y_{s-})^{2}+2\int_{0}^{t}e^{\beta
s}Y_{s}d(B)_{s+}+\sum_{0\leq s<t}e^{\beta
  s}(Y_{s+}-Y_{s})^{2}
\end{eqnarray*}
\end{corollary}
\begin{proof}
For the corollary demonstration, it suffices to apply the change of
variables formula from Theorem$~\ref{galchouklenglart}$ with $n=2$,
$F(x,y)=xy^{2}$, $X_{t}^{1}= e^{\beta t}$ and  $X_{t}^{2}=Y_{t}$.
\end{proof}
\begin{corollary}
\label{corollaire2}
 Let $Y$ be a one-dimensional optional
semimartingale with decomposition $Y_t=Y_0+M_t+A_t+B_t$, where $M$,
$A$ and $B$ are as in the above theorem. Then, almost surely, for all $t$ in $[0,T]$,
\begin{eqnarray*}
  |Y_{t}|^{2}&+&\int_{t}^{T}1_{\{Y_{s}\geq0\}}d<M^{c},M^{c}>_{s}+J^{+}_{T}(2)-J^{+}_{t}(2)+J^{-}_{T}(2)-J^{-}_{t}(2)\\
  &\leq &  |Y_{T}|^{2}+2\int_{t}^{T}|Y_{s-}|1_{\{Y_{s-}\geq0\}}d(A+M)_{s}+2\sum_{t\leq s\leq T}|Y_{s}|1_{\{Y_{s}\geq0\}}(Y_{s+}-Y_{s})
\end{eqnarray*}
where $J^{+}_{t}(2)= \sum_{s<t}\Bigl(|Y_{s+}|^{2}-|Y_{s}|^{2}-2|Y_{s}|1_{\{Y_{s}\geq0\}}(Y_{s+}-Y_{s})\Bigr)$ and
$J^{-}_{t}(2)= \sum_{s\leq t}\Bigl(|Y_{s}|^{2}-|Y_{s-}|^{2}-2|Y_{s-}|1_{\{Y_{s-}\geq0\}}(Y_{s}-Y_{s-})\Bigr)$ for all $t\in[0,T]$.
\end{corollary}
\begin{proof}
 Follows from Theorem$~\ref{galchouklenglart}$ and  \cite[Corollary 5.5]{Klimsiak2018}
\end{proof}

Now, let us prove the Theorem \ref{lemme3}.
\begin{proof}[Proof of Theorem \ref{lemme3}]
\label{proofoflemme3}
For all  $S\in\mathcal{T}_{0,T}$, we define the family $\overline{Y}(S)$ by:
\begin{equation}
\overline{Y}(S)=ess\sup_{\tau\in\mathcal{T}_{S,T}}E\Bigl[\xi_{\tau}+\int_{S}^{\tau}
f(t)dt|\mathcal{F}_{S} \Bigr] \,\,\,\  , \,\,\,\
\overline{Y}(T)=\xi_{T} \label{mono21}
\end{equation}
We put
\begin{equation}
\overline{\overline{Y}}(S)=\overline{Y}(S)+\int_{0}^{S}f(t)dt=ess\sup_{\tau\in\mathcal{T}_{S,T}}E\Bigl[\xi_{\tau}+\int_{0}^{\tau}
f(t)dt|\mathcal{F}_{S} \Bigr]
 \label{moneq29}
\end{equation}
By the Proposition  $~\ref{proposition1}$, there exists a  l\`{a}dl\`{a}g optional process $(Y_{t})_{t\in[0,T]}$ which
aggregates the family $(\overline{Y}(S))_{S\in\mathcal{T}_{0,T}}$, that is,
\begin{equation}
\overline{Y}_{S}=\overline{Y}(S)\,\,\,\ a.s. \,\,\,\ for\,\,\,\ all \,\,\ S\in\mathcal{T}_{0,T}
 \label{mono22}
\end{equation}

\textbf{Step 1: let us show that $\overline{Y}\in\mathds{S}^{2,0}$.}  By using the definition of $\overline{Y}$
$~\eqref{mono20}$, Jensen's inequality and the triangular
inequality, we get
$$|\overline{Y}_{S}|\leq ess\sup_{\tau\in\mathcal{T}_{S,T}}E\Bigl[|\xi_{\tau}|+|\int_{S}^{\tau}
f(t)dt| |\mathcal{F}_{S} \Bigr]\leq
E\Bigl[ess\sup_{\tau\in\mathcal{T}_{S,T}}|\xi_{\tau}|+\int_{0}^{T}
|f(t)|dt |\mathcal{F}_{S} \Bigr]$$ Thus, we obtain
\begin{equation}
|\overline{Y}_{S}|\leq E\Bigl[X|\mathcal{F}_{S}\Bigr]
 \label{moneq30}
\end{equation}
with
\begin{equation}
X=\int_{0}^{T}|f(t)|dt+ess\sup_{\tau\in\mathcal{T}_{0,T}}|\xi_{\tau}|.
 \label{moneq31}
\end{equation}
Applying Cauchy-Schwarz inequality, we get
\begin{equation}
E[X^2]\leq cTE\Bigl[\int_{0}^{T}f(s)^{2}ds\Bigr]+c\||\xi|\|_{\mathds{S}^{2,0}}^{2}<\infty.
\label{moneq32}
\end{equation}
where $c$ is a positive constant. Now, the inequality $~\eqref{moneq30}$
leads to $|\overline{Y}_{S}|^{2}\leq |E[X|\mathcal{F}_{S}]|^{2}$. By
taking the essential supremum over $S\in\mathcal{T}_{0,T}$ we get
$ess\sup_{S\in\mathcal{T}_{0,T}}|\overline{Y}_{S}|^{2}\leq
ess\sup_{S\in\mathcal{T}_{0,T}}|E[X|\mathcal{F}_{S}]|^{2}$. By using the
Proposition A.3 in \cite{Ouknine2015}, we get
$ess\sup_{S\in\mathcal{T}_{0,T}}|\overline{Y}_{S}|^{2}\leq
\sup_{t\in[0,T]}|E[X|\mathcal{F}_{t}]|^{2}$. By using this
inequality and Doob's martingale inequalities, we obtain
\begin{equation}
E\Bigl[ess\sup_{S\in\mathcal{T}_{0,T}}|\overline{Y}_{S}|^{2}\Bigr]\leq
E\Bigl[\sup_{t\in[0,T]}|E[X|\mathcal{F}_{t}]|^{2}\Bigr] \leq
cE[X^{2}] \label{moneq33}
\end{equation}
where $c$ is a positive constant  that changes from line to line.
Finally, combining the inequalities $~\eqref{moneq32}$ and
$~\eqref{moneq33}$, we get
\begin{equation}
E\Bigl[ess\sup_{S\in\mathcal{T}_{0,T}}|\overline{Y}_{S}|^{2}\Bigr]\leq
cTE\Bigl[\int_{0}^{T}f(s)^{2}ds\Bigr]+c\||\xi|\|_{\mathds{S}^{2,0}}^{2}<\infty.
\label{moneq34}
\end{equation}
Then $\overline{Y}_{S}\in\mathds{S}^{2,0}$. \\

\textbf{Step 2: the existence of $Z$, $A$ and $C$.} By the previous step and since $E\Bigl[\int_{0}^{T}f^{2}(t)dt\Bigr]<\infty$, the strong optional supermartingale
$\overline{\overline{Y}}$ is of class $(D)$. Applying Mertens
decomposition (see \cite[Theorem A.1]{Ouknine2015}) and a result from
optimal stopping theory (see more in \cite[Proposition 2.34. page 131]{ElKaroui1981} or \cite{KobylanskiQuenez2012}), we have the following
$$\overline{\overline{Y}}_{\tau}=M_{\tau}-A_{\tau}-C_{\tau-} \,\,\,\ \forall \tau\in\mathcal{T}_{0,T}$$
and by $~\eqref{moneq29}$ we obtain
\begin{equation}
\overline{Y}_{\tau}=-\int_{0}^{\tau}f(t)dt+M_{\tau}-A_{\tau}-C_{\tau-}
\,\,\,\ a.s. \,\,\,\ \forall \tau\in\mathcal{T}_{0,T}
 \label{moneq35}
\end{equation}
where $M$ is a (c\`{a}dl\`{a}g) uniformly integrable martingale such that
$M_0=0$, $A$ is a nondecreasing right-continuous predictable process
such that $A_0=0$, $E(A_T)<\infty$ and satisfying $~\eqref{moneq6}$,
and $C$ is a nondecreasing right-continuous adapted purely
discontinuous process such that $C_{0-}=0$, $E(C_T)<\infty$ and
satisfying $~\eqref{moneq7}$. By the martingale representation
theorem (Lemma$~\ref{lemme1}$), there exists a unique predictable
process $Z$ such that
$$M_{t}=\int_{0}^{t}\int_{\mathcal{U}}Z_{s}(x)(\mu-\nu)(ds,dx).$$
Moreover, we have $\overline{Y}_{T}=\xi_{T}$ a.s. by definition of
$\overline{Y}$. Combining this with equation $~\eqref{moneq35}$ we get the equation $~\eqref{moneq4}$. Also by definition of
$\overline{Y}$, we have $\overline{Y}_{S}\geq\xi_{S}$ a.s. for all
$S\in\mathcal{T}_{0,T}$, which, along with Proposition $~\ref{AshkanNikeghbali}$
 (or  with \cite[Theorem 3.2.]{AshkanNikeghbali2006}), shows that $\overline{Y}_{t}\geq \xi_{t}$, $0\leq t \leq T$ a.s. Finally, to conclude that the process
$(\overline{Y},Z,A,C)$ is a solution to the reflected BSDE with
parameters $(f,\xi)$, it remains to show that $Z\times A\times C\in\mathds{H}^{2,0}\times\mathds{S}^{2,0}\times\mathds{S}^{2,0}$.\\

\textbf{Step 3: let us prove that $A\times C\in\mathds{S}^{2,0}\times\mathds{S}^{2,0}$.}
Let us define the process $\overline{\overline{A}}_{t}=A_{t}+C_{t-}$
where the processes $A$ and  $C$ are given by $~\eqref{moneq35}$. By similar
arguments as those used in the proof of inequality
$~\eqref{moneq30}$, we see that $|\overline{\overline{Y}}_{S}|\leq
E[X|\mathcal{F}_{S}]$ with
$$X=\int_{0}^{T}|f(t)|dt+ess\sup_{\tau\in\mathcal{T}_{S,T}}|\xi_{\tau}|.$$
Then, the \cite[Corollary A.1]{Ouknine2015} ensures the
existence of a constant $c>0$ such that
$E[(\overline{\overline{A}}_{T})^{2}]\leq cE[X^{2}]$. By combining
this inequality with inequality $~\eqref{moneq32}$, we obtain
\begin{equation}
E[(\overline{\overline{A}}_{T})^{2}]\leq
cTE\Bigl[\int_{0}^{T}f(s)^{2}ds\Bigr]+c\||\xi|\|_{\mathds{S}^{2,0}}^{2}
 \label{moneq36}
\end{equation}
where we have again allowed the positive constant $c$ to vary from
line to line. We conclude that  $\overline{\overline{A}}\in L^{2}$.
According to the nondecreasingness of $\overline{\overline{A}}$, we have
$(\overline{\overline{A}}_{\tau})^{2}\leq
(\overline{\overline{A}}_{T})^{2}$ for all
$\tau\in\mathcal{T}_{0,T}$ thus
$$E\Bigl[ess\sup_{\tau\in\mathcal{T}_{0,T}}(\overline{\overline{A}}_{\tau})^{2}\Bigr]\leq
E\Bigl[(\overline{\overline{A}}_{T})^{2}\Bigr]$$ i.e.
$\overline{\overline{A}}\in \mathds{S}^{2}$ then $A\in
\mathds{S}^{2,0}$ and $C\in \mathds{S}^{2,0}$. \\

\textbf{Step 4: we show now that $Z\in\mathds{H}^{2,0}$.}
We have from \textbf{step 3}
$$\int_{0}^{T}\int_{\mathcal{U}}Z_{s}(x)(\mu-\nu)(ds,dx)=
\overline{Y}_{T}+\int_{0}^{T}f(t)dt+\overline{\overline{A}}_{T}-\overline{Y}_{0}$$
where $\overline{\overline{A}}$ is the process from \textbf{Step 3}. Since
$\overline{\overline{A}}_{T}\in L^{2}$, $\overline{Y}_{T}\in L^{2}$,
 $\overline{Y}_{0}\in L^{2}$ and $E\Bigl[\int_{0}^{T}f^{2}(t)dt\Bigr]<\infty$. Hence,  $\int_{0}^{T}\int_{\mathcal{U}}Z_{s}(x)(\mu-\nu)(ds,dx)\in L^{2}$ and consequently
$Z\in\mathds{H}^{2,0}$.\\

  For the uniqueness of the solution, suppose that $(Y,Z,A,C)$
is a solution of the reflected BSDE with driver $f$ and obstacle $\xi$. Then,
by Lemma$~\ref{lemme2}$ (applied with $f_{1}=f_{2}=f$) we obtain
$Y=\overline{Y}$ in $\mathds{S}^{2,0}$, where $\overline{Y}$ is given
by $~\eqref{mono21}$. The uniqueness of $A$, $C$ and $Z$ follows from the uniqueness of Mertens decomposition of strong optional supermartingales and from the uniqueness of the martingale
representation (Lemma$~\ref{lemme1}$).

\end{proof}


\bibliographystyle{alea3}
\bibliography{references}

\begin{thebibliography}{32}
\providecommand{\natexlab}[1]{#1}
\providecommand{\url}[1]{\texttt{#1}}
\providecommand{\urlprefix}{URL }
\expandafter\ifx\csname urlstyle\endcsname\relax
  \providecommand{\doi}[1]{doi:\discretionary{}{}{}#1}\else
  \providecommand{\doi}{doi:\discretionary{}{}{}\begingroup
  \urlstyle{rm}\Url}\fi
\providecommand{\eprint}[2][]{\url{#2}}

\bibitem[{Akdim et~al.(2020)Akdim, Haddadi and Ouknine}]{akdim}
K.~Akdim, M.~Haddadi and Y.~Ouknine.
\newblock Strong snell envelopes and rbsdes with regulated trajectories when
  the barrier is a semimartingale.
\newblock \emph{Stochastics} \textbf{92}, 335--355 (2020).

\bibitem[{Baadi and Ouknine(2017)}]{baadi2017}
B.~Baadi and Y.~Ouknine.
\newblock {R}eflected {BSDE}s when the obstacle is not right-continuous in a
  general filtration.
\newblock \emph{{ALEA}, Lat. Am. J. Probab. Math. Stat.} \textbf{14}, 201--218
  (2017).

\bibitem[{Baadi and Ouknine(2018)}]{baadi2018}
B.~Baadi and Y.~Ouknine.
\newblock {R}eflected {BSDE}s with optional barrier in a general filtration.
\newblock \emph{Afrika Matematika} \textbf{29}, 1049--1064 (2018).
\newblock \urlprefix\url{https://doi.org/10.1007/s13370-018-0600-6}.

\bibitem[{Barles et~al.(1997)Barles, Bukhdan and Pardoux}]{BBP:1997}
G.~Barles, R.~Bukhdan and E.~Pardoux.
\newblock Backward stochastic differential equations and integral-partial
  differential equations.
\newblock \emph{Stochastics and stochastics reports} \textbf{57--83}, 57--83
  (1997).

\bibitem[{Dellacherie and Meyer(1975)}]{DellacherieMeyer1975}
C.~Dellacherie and P.-A. Meyer.
\newblock \emph{{P}robabilitié et {P}otentiel}.
\newblock {H}ermann (1975).
\newblock {C}hap. {I-IV}, {N}ouvelle {\'e}dition.

\bibitem[{Dellacherie and Meyer(1980)}]{DellacherieMeyer1980}
C.~Dellacherie and P.-A. Meyer.
\newblock \emph{{P}robabilit{\'e}s et {P}otentiel, {T}h{\'e}orie des
  {M}artingales}.
\newblock {H}ermann (1980).
\newblock {C}hap. {V-VIII}, {N}ouvelle {\'e}dition.

\bibitem[{El~Karoui(1981)}]{ElKaroui1981}
N.~El~Karoui.
\newblock {L}es aspects probabilistes du contr\^{o}le stochastique, {E}cole
  d{'E}t{\'e} de {P}robabilit{\'e}s de {S}aint-{F}lour {IX-1979}.
\newblock \emph{{L}ect. {N}otes in {M}ath.} \textbf{876}, 73--238 (1981).

\bibitem[{El~Karoui and Hamad\`{e}ne(2003)}]{EH:2003}
N.~El~Karoui and S.~Hamad\`{e}ne.
\newblock Bsdes and risk-sensitive control, zero-sum and nonzero-sum game
  problems of stochastic functional differential equations.
\newblock \emph{Stochastic Processes and their Applications} \textbf{107},
  145--169 (2003).

\bibitem[{El~Karoui et~al.(1997{\natexlab{a}})El~Karoui, Kapoudjian, Pardoux,
  Peng and Quenez}]{ELKaroui1997}
N.~El~Karoui, C.~Kapoudjian, E.~Pardoux, S.~Peng and M.-C. Quenez.
\newblock {R}eflected solutions of backward {SDE}'s, and related obstacle
  problems for {PDE}'s.
\newblock \emph{The {A}nnals of {P}robability} \textbf{25}~(2), 702--737
  (1997{\natexlab{a}}).

\bibitem[{El~Karoui et~al.(1997{\natexlab{b}})El~Karoui, Peng and
  Quenez}]{ELKarouiPeng1997}
N.~El~Karoui, S.~Peng and M.-C. Quenez.
\newblock {B}ackward stochastic differential equations in finance.
\newblock \emph{{M}athematical {F}inance} \textbf{7}~(1), 1--71
  (1997{\natexlab{b}}).

\bibitem[{El~Karoui and Quenez(1997)}]{EQ:1997}
N.~El~Karoui and M.~C. Quenez.
\newblock {Non-linear pricing theory and backward stochastic differential
  Equations}.
\newblock \emph{Financial mathematics} \textbf{1656}, 191--246 (1997).

\bibitem[{Essaky(2008)}]{Essaky2008}
E-H. Essaky.
\newblock {R}eflected backward stochastic differential equation with jumps and
  {RCLL} obstacle.
\newblock \emph{{B}ulletin des {S}ciences {M}ath{\'e}matiques}
  \textbf{132}~(8), 690--710 (2008).

\bibitem[{Foresta(2019)}]{foresta2019}
N.~Foresta.
\newblock Optimal stopping of marked point processes and reflected backward
  stochastic differential equations.
\newblock \emph{Applied Mathematics and Optimization}  (2019).
\newblock \urlprefix\url{https://doi.org/10.1007/s00245-019-09585-y}.

\bibitem[{Gal'chouk(1981)}]{Galchouk1981}
L.I. Gal'chouk.
\newblock {O}ptional martingales.
\newblock \emph{{M}ath. {USSR} {S}bornik} \textbf{40}~(4), 435--468 (1981).

\bibitem[{Grigorova et~al.(2020)Grigorova, Imkeller, Ouknine and Quenez}]{GIOQ}
M.~Grigorova, E.~Imkeller, Y.~Ouknine and M.~C. Quenez.
\newblock Optimal stopping with $f$-expectations: The irregular case.
\newblock \emph{Stochastic Processes and their Applications} \textbf{130},
  1258--1288 (2020).

\bibitem[{Grigorova et~al.(2017)Grigorova, Imkeller, Offen, Ouknine and
  Quenez}]{Ouknine2015}
M.~Grigorova, P.~Imkeller, E.~Offen, Y.~Ouknine and M.-C. Quenez.
\newblock {R}eflected {BSDE}s when the obstacle is not right-continuous and
  optimal stopping.
\newblock \emph{Ann. Appl. Probab} \textbf{27}~(5), 3153--3188 (2017).

\bibitem[{Hamad\`{e}ne(2002)}]{Hamadene2002}
S.~Hamad\`{e}ne.
\newblock {R}eflected {BSDE}'s with discontinuous barrier and application.
\newblock \emph{Stochastics and Stochastic Reports} \textbf{74}~(3-4), 571--596
  (2002).

\bibitem[{Hamad\`{e}ne and Lepeltier(1995)}]{HL}
S.~Hamad\`{e}ne and J-P. Lepeltier.
\newblock Zero-sum stochastic differential games and backward equations.
\newblock \emph{Systems and Control Letters} \textbf{24(4)}, 259--263 (1995).

\bibitem[{Hamad\`{e}ne and Ouknine(2003)}]{HamadeneOknine2003}
S.~Hamad\`{e}ne and Y.~Ouknine.
\newblock {B}ackward stochastic differential equations with jumps and random
  obstacle.
\newblock \emph{{E}lectronic {J}ournal of {P}robability} \textbf{8}~(2), 1--20
  (2003).

\bibitem[{Hamad\`{e}ne and Ouknine(2016)}]{HamadeneOknine2011}
S.~Hamad\`{e}ne and Y.~Ouknine.
\newblock {R}eflected {B}ackward {SDE}s with general jumps.
\newblock \emph{{T}heory of {P}robability and {I}ts {A}pplications}
  \textbf{60}~(2), 263--280 (2016).

\bibitem[{Jacod(1975)}]{Jacod1975}
J.~Jacod.
\newblock Multivariate point processes: predictable projection, radon-nikodym
  derivatives, representation of martingales.
\newblock \emph{Z. Wahrscheinlichkeitstheorie und Verw. Gebiete,} \textbf{31},
  235–253 (1975).

\bibitem[{Klimsiak et~al.(2019)Klimsiak, Rzymowski and
  Slominski}]{Klimsiak2018}
T.~Klimsiak, M.~Rzymowski and L.~Slominski.
\newblock {R}eflected {BSDE}s with regulated trajectories.
\newblock \emph{Stochastic Processes and their Applications} \textbf{129}~(4),
  1153--1184 (2019).

\bibitem[{Kobylanski and Quenez(2012)}]{KobylanskiQuenez2012}
M.~Kobylanski and M.-C. Quenez.
\newblock {O}ptimal stopping time problem in a general framework.
\newblock \emph{{E}lectronic {J}ournal of {P}robability} \textbf{17}, 1--28
  (2012).

\bibitem[{Lenglart(1980)}]{Lenglart1980}
E.~Lenglart.
\newblock {T}ribus de {M}eyer et th{\'e}orie des processus, {S}{\'e}minaire de
  probabilit{\'e}s de {S}trasbourg {XIV} 1978/79.
\newblock \emph{{L}ecture {N}otes in {M}athematics} \textbf{784}, 500--546
  (1980).

\bibitem[{Maingueneau(1978)}]{Maingueneau1978}
M.~A. Maingueneau.
\newblock Temps d'arrêt optimaux et théorie générale.
\newblock \emph{Séminaire de probabilités de Strasbourg} \textbf{12},
  457--467 (1978).
\newblock \urlprefix\url{http://eudml.org/doc/113167}.

\bibitem[{Marzougue(2020)}]{Ma}
M.~Marzougue.
\newblock A note on optional snell envelopes and reflected backward sdes.
\newblock \emph{Statistics \& probability letters} page 108833 (2020).
\newblock \urlprefix\url{https://doi.org/10.1016/j.spl.2020.108833}.

\bibitem[{Marzougue and El~Otmani(2019)}]{ME:2019}
M.~Marzougue and M.~El~Otmani.
\newblock Bsdes with right upper-semicontinuous reflecting obstacle and
  stochastic lipschitz coefficient.
\newblock \emph{Random Operators and Stochastic Equations} \textbf{27}~(1),
  27--41 (2019).

\bibitem[{Marzougue and El~Otmani(2020)}]{ME:2020}
M.~Marzougue and M.~El~Otmani.
\newblock {Predictable solution for reflected BSDEs when the obstacle is not
  right-continuous}.
\newblock \emph{Random Operators and Stochastic Equations} \textbf{28},
  269--279 (2020).

\bibitem[{Nikeghbali(2006)}]{AshkanNikeghbali2006}
A.~Nikeghbali.
\newblock {A}n essay on the general theory of stochastic processes.
\newblock \emph{{P}robab. {S}urv} \textbf{3}, 345--412 (2006).

\bibitem[{Pardoux and Peng(1990)}]{PP1990}
E.~Pardoux and S.~Peng.
\newblock {A}dapted solution of a backward stochastic differential equation.
\newblock \emph{{S}ystems and {C}ontrol {L}etters} \textbf{14}, 55--61 (1990).

\bibitem[{Pardoux and Peng(1992)}]{PP1992}
E.~Pardoux and S.~Peng.
\newblock {B}ackward stochastic differential equations and quasilinear
  parabolic partial differential equations.
\newblock \emph{{L}ect. {N}otes in {CIS}176} pages 200--217 (1992).

\bibitem[{Protter(2000)}]{Protter2000}
P.E. Protter.
\newblock \emph{{S}tochastic integration and differential equations}.
\newblock {S}pringer (2000).
\newblock {V}ersion 2.1.

\end{thebibliography}


\begin{thebibliography}{14}
\providecommand{\natexlab}[1]{#1}
\providecommand{\url}[1]{\texttt{#1}}
\providecommand{\urlprefix}{URL }
\expandafter\ifx\csname urlstyle\endcsname\relax
  \providecommand{\doi}[1]{doi:\discretionary{}{}{}#1}\else
  \providecommand{\doi}{doi:\discretionary{}{}{}\begingroup
  \urlstyle{rm}\Url}\fi
\providecommand{\eprint}[2][]{\url{#2}}

\bibitem[{Baadi and Ouknine(2017)}]{baadi2017}
B.~Baadi and Y.~Ouknine.
\newblock {R}eflected {BSDE}s when the obstacle is not right-continuous in a
  general filtration.
\newblock \emph{{ALEA}, Lat. Am. J. Probab. Math. Stat.} \textbf{14}, 201--218
  (2017).

\bibitem[{Baadi and Ouknine(2018)}]{baadi2018}
B.~Baadi and Y.~Ouknine.
\newblock {R}eflected {BSDE}s with optional barrier in a general filtration.
\newblock \emph{Afrika Matematika} \textbf{29}, 1049--1064 (2018).
\newblock \urlprefix\url{https://doi.org/10.1007/s13370-018-0600-6}.

\bibitem[{Dellacherie and Meyer(1975)}]{DellacherieMeyer1975}
C.~Dellacherie and P.-A. Meyer.
\newblock \emph{{P}robabilitié et {P}otentiel}.
\newblock {H}ermann (1975).
\newblock {C}hap. {I-IV}, {N}ouvelle {\'e}dition.

\bibitem[{Dellacherie and Meyer(1980)}]{DellacherieMeyer1980}
C.~Dellacherie and P.-A. Meyer.
\newblock \emph{{P}robabilit{\'e}s et {P}otentiel, {T}h{\'e}orie des
  {M}artingales}.
\newblock {H}ermann (1980).
\newblock {C}hap. {V-VIII}, {N}ouvelle {\'e}dition.

\bibitem[{El~Karoui(1981)}]{ElKaroui1981}
N.~El~Karoui.
\newblock {L}es aspects probabilistes du contr\^{o}le stochastique, {E}cole
  d{'E}t{\'e} de {P}robabilit{\'e}s de {S}aint-{F}lour {IX-1979}.
\newblock \emph{{L}ect. {N}otes in {M}ath.} \textbf{876}, 73--238 (1981).

\bibitem[{Gal'chouk(1981)}]{Galchouk1981}
L.I. Gal'chouk.
\newblock {O}ptional martingales.
\newblock \emph{{M}ath. {USSR} {S}bornik} \textbf{40}~(4), 435--468 (1981).

\bibitem[{Grigorova et~al.(2017)Grigorova, Imkeller, Offen, Ouknine and
  Quenez}]{Ouknine2015}
M.~Grigorova, P.~Imkeller, E.~Offen, Y.~Ouknine and M.-C. Quenez.
\newblock {R}eflected {BSDE}s when the obstacle is not right-continuous and
  optimal stopping.
\newblock \emph{Ann. Appl. Probab} \textbf{27}~(5), 3153--3188 (2017).

\bibitem[{Jacod(1975)}]{Jacod1975}
J.~Jacod.
\newblock Multivariate point processes: predictable projection, radon-nikodym
  derivatives, representation of martingales.
\newblock \emph{Z. Wahrscheinlichkeitstheorie und Verw. Gebiete,} \textbf{31},
  235–253 (1975).

\bibitem[{Klimsiak et~al.(2019)Klimsiak, Rzymowski and
  Slominski}]{Klimsiak2018}
T.~Klimsiak, M.~Rzymowski and L.~Slominski.
\newblock {R}eflected {BSDE}s with regulated trajectories.
\newblock \emph{Stochastic Processes and their Applications} \textbf{129}~(4),
  1153--1184 (2019).

\bibitem[{Kobylanski and Quenez(2012)}]{KobylanskiQuenez2012}
M.~Kobylanski and M.-C. Quenez.
\newblock {O}ptimal stopping time problem in a general framework.
\newblock \emph{{E}lectronic {J}ournal of {P}robability} \textbf{17}, 1--28
  (2012).

\bibitem[{Lenglart(1980)}]{Lenglart1980}
E.~Lenglart.
\newblock {T}ribus de {M}eyer et th{\'e}orie des processus, {S}{\'e}minaire de
  probabilit{\'e}s de {S}trasbourg {XIV} 1978/79.
\newblock \emph{{L}ecture {N}otes in {M}athematics} \textbf{784}, 500--546
  (1980).

\bibitem[{Maingueneau(1978)}]{Maingueneau1978}
M.~A. Maingueneau.
\newblock Temps d'arrêt optimaux et théorie générale.
\newblock \emph{Séminaire de probabilités de Strasbourg} \textbf{12},
  457--467 (1978).
\newblock \urlprefix\url{http://eudml.org/doc/113167}.

\bibitem[{Nikeghbali(2006)}]{AshkanNikeghbali2006}
A.~Nikeghbali.
\newblock {A}n essay on the general theory of stochastic processes.
\newblock \emph{{P}robab. {S}urv} \textbf{3}, 345--412 (2006).

\bibitem[{Protter(2000)}]{Protter2000}
P.E. Protter.
\newblock \emph{{S}tochastic integration and differential equations}.
\newblock {S}pringer (2000).
\newblock {V}ersion 2.1.

\end{thebibliography}

\end{document}